\documentclass{amsart}
\usepackage{amssymb}
\usepackage{amsfonts}
\usepackage{amssymb}
\usepackage{amsmath}
\usepackage{amsthm}
\usepackage{enumerate}
\usepackage{tabularx}
\usepackage{centernot}
\usepackage{mathtools}
\usepackage{stmaryrd}
\usepackage{amsthm,amssymb}
\usepackage{etoolbox}
\usepackage{tikz}
\usepackage{amssymb}
\usetikzlibrary{matrix}
\usepackage{tikz-cd}
\numberwithin{equation}{section}

\usepackage{tikz}
\usetikzlibrary{decorations.pathmorphing,shapes}

\newcounter{sarrow}

\renewcommand{\leq}{\leqslant}
\renewcommand{\geq}{\geqslant}

\makeatletter
\def\subsection{\@startsection{subsection}{3}%
  \z@{.5\linespacing\@plus.7\linespacing}{.3\linespacing}%
  {\bfseries\centering}}
\makeatother

\makeatletter
\def\subsubsection{\@startsection{subsubsection}{3}%
  \z@{.5\linespacing\@plus.7\linespacing}{.3\linespacing}%
  {\centering}}
\makeatother

\makeatletter
\def\myfnt{\ifx\protect\@typeset@protect\expandafter\footnote\else\expandafter\@gobble\fi}
\makeatother

\theoremstyle{definition}

\newtheorem{theorem}{Theorem}[section]
\newtheorem{definition}[theorem]{Definition}
\newtheorem{lemma}[theorem]{Lemma}
\newtheorem{proposition}[theorem]{Proposition}

\newtheorem{corollary}[theorem]{Corollary}
\newtheorem{oproblem}[theorem]{Open Problem}

\newtheorem{remark}[theorem]{Remark}

\newcounter{claimcounter}
\numberwithin{claimcounter}{theorem}

\newcommand{\pureindep}[1][]{%
  \mathrel{
    \mathop{
      \vcenter{
        \hbox{\oalign{\noalign{\kern-.3ex}\hfil$\vert$\hfil\cr
              \noalign{\kern-.7ex}
              $\smile$\cr\noalign{\kern-.3ex}}}
      }
    }\displaylimits_{#1}
  }
}

\newcommand{\indep}[2]{%
  \mathrel{
    \mathop{
      \vcenter{
        \hbox{%
\oalign{
\noalign{\kern-.3ex}\hfil$\vert$\hfil\cr
              \noalign{\kern-.7ex}
              $\smile$\cr\noalign{\kern-.3ex}
}
}
      }
}^{\!\!\!\!\!#2}_{\!\!\hspace{-0.1em}#1}
  }
}

\newcommand{\displayindep}[2]{%
  \mathrel{
    \mathop{
      \vcenter{
        \hbox{%
\oalign{
\noalign{\kern-.3ex}\hfil$\vert$\hfil\cr
              \noalign{\kern-.7ex}
              $\smile$\cr\noalign{\kern-.3ex}
}
}
      }
}^{\!\!\hspace{-0.1em}#2}_{\!\!\hspace{-0.1em}#1}
  }
}

\newcommand{\displayfindep}[2]{%
  \mathrel{
    \mathop{
      \vcenter{
        \hbox{%
\oalign{
\noalign{\kern-.3ex}\hfil$\vert$\hfil\cr
              \noalign{\kern-.7ex}
              $\smile$\cr\noalign{\kern-.3ex} 
}
}
      }
}^{\!\hspace{-0.14em}#2}_{\!\!\hspace{-0.05em}#1}
  }
}

\newcommand{\card}[1]{|{#1}|}

\def\presuper#1#2%
  {\mathop{}%
   \mathopen{\vphantom{#2}}^{#1}%
   \kern-\scriptspace%
   #2}

\begin{document}

\begin{abstract} We study classes of right-angled Coxeter groups with respect to the strong submodel relation of parabolic subgroup. We show that the class of all right-angled Coxeter group is not smooth, and establish some general combinatorial criteria for such classes to be abstract elementary classes, for them to be finitary, and for them to be tame. We further prove two combinatorial conditions ensuring the strong rigidity of a right-angled Coxeter group of arbitrary rank. The combination of these results translate into a machinery to build concrete examples of $\mathrm{AECs}$ satisfying given model-theoretic properties. We exhibit the power of our method constructing three concrete examples of finitary classes. We show that the first and third class are non-homogeneous, and that the last two are tame, uncountably categorical and axiomatizable by a single $L_{\omega_{1}, \omega}$-sentence. We also observe that the isomorphism relation of any countable complete first-order theory is $\kappa$-Borel reducible (in the sense of generalized descriptive set theory) to the isomorphism relation of the theory of right-angled Coxeter groups whose Coxeter graph is an infinite random graph.
%
%
\end{abstract}

\title[Coxeter Groups and $\mathrm{AECs}$: the Right-Angled Case]{Coxeter Groups and Abstract Elementary Classes: the Right-Angled Case}
\thanks{The research of the second author was supported by the Finnish Academy of Science and Letters (Vilho, Yrj\"o and Kalle V\"ais\"al\"a foundation). The authors would like to thank John Baldwin for interesting remarks on this paper.}

\author{Tapani Hyttinen}
\address{Department of Mathematics and Statistics,  University of Helsinki, Finland}

\author{Gianluca Paolini}
\address{Einstein Institute of Mathematics,  The Hebrew University of Jerusalem, Israel}

\date{\today}
\maketitle


\section{Introduction}

	Abstract elementary classes ($\mathrm{AECs}$) \cite{shelah_abstr_ele_cla} are pairs $(\mathbf{K}, \preccurlyeq)$ such that $\mathbf{K}$ is a class of structures of the same similarity type, and $\preccurlyeq$ is a partial order on $\mathbf{K}$, often referred to as a {\em strong submodel relation}, satisfying a certain set of axioms, which generalise some of the properties of the relation of elementary submodel of first-order logic. Although $\mathrm{AECs}$ generalize the first-order setting, the situation in $\mathrm{AECs}$ is very different from the one in elementary model theory. In fact, in the latter setting the strong submodel relation is {\em always fixed}. The same remark holds for the model theory of infinitary languages, since also in this context one tends to use the canonical strong submodel relations (which in this case depend on what is the formula defining the class under study). On the other hand, in the theory of abstract elementary classes we are free to choose {\em any} strong submodel relation, as long as the $\mathrm{AECs}$ axioms are satisfied. This choice determines very strongly the model-theoretic properties of the class under analysis. A classical example is when we consider as $\mathbf{K}$ the class of all abelian groups. In this case, letting $\preccurlyeq_0$ to be the subgroup relation, and $\preccurlyeq_{1}$ to the pure subgroup relation, we have that $(\mathbf{K}, \preccurlyeq_0)$ is $\omega$-stable, while $(\mathbf{K}, \preccurlyeq_{1})$ is not even superstable.

	In the context of $\mathrm{AECs}$, when one tries to find examples of
various model-theoretic properties, one tends to start from a class $\mathbf{K}$ of structures, and then search for a suitable or natural strong submodel relation $\preccurlyeq$. In this paper we make an experiment, and reverse this process. That is, we first choose the relation $\preccurlyeq$ and then we try to find $\mathbf{K}$ so that $(\mathbf{K}, \preccurlyeq)$ satisfies certain given model-theoretic properties.  We hope that in this way we are able to increase our understanding of the vast number of dividing lines that currently dominate the universe of $\mathrm{AECs}$, and to generate new (counter-)examples for the theory. A similar approach has been pioneered in \cite{geometric_lattices}, where several well-behaved classes of geometric lattices have been found in this way, when considering as $\preccurlyeq$ the strong submodel relation of principal extension of a combinatorial geometry, arising from the work of Crapo \cite{crapo}. 

	In this case study, we consider the strong submodel relation of {\em parabolic subgroup}, from {\em geometric group theory}. The beginning of our study is the search for groups which together with the parabolic subgroup relation are $\mathrm{AECs}$ (i.e. the first property we test is the property of being an $\mathrm{AEC}$). We very quickly restricted our attention to classes consisting  of so-called {\em right-angled Coxeter groups}. These groups are in fact the most well-understood structures in geometric group theory. In particular, they satisfy a crucial requirement known as {\em rigidity} \cite{castella} (see  below\footnote{Notice that here rigidity does not mean what it usually means in model theory.}). However, it turns out that rigidity alone is not enough for our purposes. In fact, we will see that the Smoothness Axiom fails in general and, without additional assumptions, we do not even know whether $\preccurlyeq$ is transitive or not. We get out of this {\em empasse} assuming a stronger property, known as {\em strong rigidity}.

	While in the case of finitely generated right-angled Coxeter groups clear necessary and sufficient conditions are known for strong rigidity, not much is known about infinitely generated ones. What is known is basically just that in this more general setting these conditions are only necessary, but not sufficient. Thus, we start our study by giving two combinatorial conditions ensuring the strong rigidity of an arbitrary right-angled Coxeter group. These results will be used to construct three concrete examples of $\mathrm{AECs}$: $(\mathbf{K}_0, \preccurlyeq)$, $(\mathbf{K}_1, \preccurlyeq)$ and $(\mathbf{K}_2, \preccurlyeq)$.

	We continue our study by giving some general criteria for a class of strongly rigid right-angled Coxeter groups to be an abstract elementary class, and for it to satisfy the usual sufficient conditions for the construction of a monster model, i.e. amalgamation, joint embedding and arbitrarly large models. We then turn to notions that describe the behaviour of
Galois-types, namely {\em homogeneity}, {\em finitarity} and {\em tameness} (we will also point out that excluding the class of infinite vector spaces over the two element field, classes of infinite right-angled Coxeter groups are not first-order axiomatizable). Also in this case we give general criteria for the satisfaction of these properties, under the assumption of strong rigidity. The underlying theme of these general results is the reduction of model-theoretic properties of a class of right-angled Coxeter groups to {\em combinatorial conditions} on the associated graphs, the so-called Coxeter graphs. These conditions are often easy to realize, and, paired with our two general results on the strong rigidity of right-angled Coxeter groups, they translate into a machinery to build concrete examples of $\mathrm{AECs}$. The classes $(\mathbf{K}_0, \preccurlyeq~)$, $(\mathbf{K}_1, \preccurlyeq)$ and $(\mathbf{K}_2, \preccurlyeq)$ should be considered under this perspective, as explicit examples of this machinery. 

	We conclude the paper with a close analysis of these classes. First, we show that $(\mathbf{K}_0, \preccurlyeq)$, $(\mathbf{K}_1, \preccurlyeq)$ and $(\mathbf{K}_2, \preccurlyeq)$ are finitary. Then, we show that $(\mathbf{K}_0, \preccurlyeq)$ has the independence property (and thus it is unstable), while $(\mathbf{K}_1, \preccurlyeq)$ and $(\mathbf{K}_2, \preccurlyeq)$ are both tame and uncountably categorical (and thus stable in every infinite cardinality). Finally, we show that $(\mathbf{K}_0, \preccurlyeq)$ and $(\mathbf{K}_2, \preccurlyeq)$ are not homogeneous. We leave the tameness of $(\mathbf{K}_0, \preccurlyeq)$ and the homogeneity of $(\mathbf{K}_1, \preccurlyeq)$ as open questions. John Baldwin pointed out to us that by combining our results with results from \cite{kueker}, various definability results can be obtained. E.g. the classes $(\mathbf{K}_1, \preccurlyeq)$ and $(\mathbf{K}_2, \preccurlyeq)$ are axiomatizable by a single $L_{\omega_{1}, \omega}$-sentence, and over strong submodels Galois types and $L_{\omega_{1}, \omega}$-types coincide in both $(\mathbf{K}_1, \preccurlyeq)$ and $(\mathbf{K}_2, \preccurlyeq)$.

	On the way of writing this paper, we also observed that right-angled
Coxeter groups provide a way of finding a group whose first-order theory is maximal in the order of complexity that was introduced in the theory of generalized descriptive set theory \cite{Fr&Hy&Ku}. We will point out how one can see this.

\section{Coxeter Groups}

	Let $S$ be a set. A matrix $m: S \times S \rightarrow \{1, 2, . . . , \infty \}$ is called a {\em Coxeter matrix} if it satisfies
$$m(s, s') = m(s' , s);$$
$$m(s, s') = 1 \Leftrightarrow s = s'.$$
	Equivalently, $m$ can be represented by a labelled graph $\Gamma$, called a {\em Coxeter graph}, whose node set is $S$ and whose edges are the unordered pairs $\{s, s' \}$ such that $m(s, s') < \infty$, with label $m(s, s')$. (Notice that some authors refer to the Coxeter graph as the graph $\Gamma$ such that $s$ and $s'$ are adjacent iff $m(s, s ) > 2$.) 
Let $S^2_{fin} = \{(s, s') \in S^2 : m(s, s' ) < \infty \}$. A Coxeter matrix $m$ determines
a group $W$ with presentation
\begin{equation}
\begin{cases} \text{Generators}: S \\
				\text{Relations}: (ss')^{m(s,s')} = e, \text{ for all } (s, s' ) \in S^2_{fin}.
\end{cases} 
\end{equation}
If a group $W$ has a presentation such as (2.1), then the pair $(W, S)$ is
called a {\em Coxeter system} of type $m = m_{(W, S)}$ or of type $\Gamma = \Gamma_{(W, S)}$. The group $W = W_{\Gamma}$ is called a {\em Coxeter group} and the set $S$ a Coxeter basis (or Coxeter generating set) for $W$. The cardinality of $S$ is called the rank of $(W, S)$. Notice that in the present paper we {\em do not} assume that our Coxeter groups are of finite rank, as it is done in most of the literature on the subject. As well-known, the isomorphism type of $\Gamma_{(W, S)}$ is not determined by the group $W$ alone (see e.g. \cite[Chapter 1, Exercise 2]{bjorner}). This motivates the following definition.

	\begin{definition}\label{def_rigidity} Let $W$ be a Coxeter group. 
	\begin{enumerate}[(1)]
	\item We say that $W$ is {\em rigid} if for any two Coxeter bases $S$ and $S'$ for $W$ there is an automorphism $\alpha \in Aut(W)$ such that $\alpha(S) = S'$.
	\item We say that $W$ is {\em strongly rigid} if for any two Coxeter bases $S$ and $S'$ for $W$ there is an {\em inner} automorphism $\alpha \in Inn(W)$ such that $\alpha(S) = S'$.
\end{enumerate}
\end{definition}

	That is, $W$ is rigid if and only if for any two Coxeter bases $S$ and $S'$ for $W$
there exists an isomorphism of labelled graphs between $\Gamma_{(W, S)}$ and $\Gamma_{(W, S')}$. The problem of deciding whether  two non-isomorphic Coxeter graphs determine isomorphic Coxeter groups is known as the {\em isomorphism problem} for Coxeter groups. This problem is highly non-trivial, and it has been solved only partially \cite{bahls}. The most well understood class of Coxeter groups in this respect (and any other respect) is the class of so-called {\em right-angled} Coxeter groups. 

	\begin{definition} We say that a Coxeter system $(W, S)$ is {\em right-angled} if $m_{(W, S)}$ has coefficients in $\{ 1, 2, \infty \}$, and that a Coxeter group $W$ is right-angled if there exists a right-angled Coxeter system for $W$. 
\end{definition}

	\begin{theorem}[Castella \cite{castella}] The right-angled Coxeter groups are rigid.
\end{theorem}

	Thus, in the case of right-angled Coxeter systems $(W, S)$ the group $W$ alone determines the isomorphism type of $\Gamma_{(W, S)}$. Consequently, given a right-angled Coxeter group $W$ we denote by $\Gamma_{W}$ (or simply $\Gamma$) its associated Coxeter graph (unique modulo graph isomorphisms). 
	Given a Coxeter group $W$ there is a special class of subgroups of $W$, which are called the {\em parabolic} subgroups of $W$. These subgroups (and the subgroup relation which they induce) will be the main ingredient in our model-theoretic analysis of right-angled Coxeter groups.
	
		\begin{definition}\label{def_parabolic} Let $W$ be a Coxeter group.
\begin{enumerate}[(1)]
\item Given a Coxeter basis $S$ for $W$, we say that $W'$ is an $S$-{\em parabolic} subgroup of $W$ if $W' = \langle S' \rangle_W$ for some $S' \subseteq S$, i.e. $W'$ is generated by a subset of $S$. In this case, we denote the subgroup $W'$ as $W_{S'}$.
\item We say that $W'$ is a {\em parabolic} subgroup of $W$, denoted as $W' \preccurlyeq W$, if $W'$ is an $S$-parabolic subgroup of $W$ for {\em some} Coxeter basis $S$ of $W$.
\end{enumerate}
\end{definition}	
	
	A parabolic subgroup $W' = \langle S' \rangle_W$ of a Coxeter group $W = (W, S)$ is Coxeter group in its own right, with as Coxeter generating set the induced subgraph determined by $S'$ (see e.g. \cite[Proposition 2.4.1]{bjorner}). As evident from the definition, the parabolic subgroup relation depends on the particular choice of Coxeter basis $S$ for $W$. This generates some difficulties in the analysis of this relation, e.g. in the proof of very basic properties such as transitivity. To this end, the notion of strong-rigidity (cfr. Definition \ref{def_rigidity}) is of great help (notice for example that in the presence of strong-rigidity the transitivity of the parabolic subgroup relation is essentially trivial, see the proof of Theorem \ref{almost_AEC_th}). For this reason we are interested in sufficient (and possibly necessary) conditions for strong rigidity.
	The problem of (strong) rigidity of a Coxeter group $W$ is of course strictly related to our understanding of the corresponding group of automorphisms $Aut(W)$. In the case of right-angled Coxeter groups a fundamental result of Tits \cite{tits} gives an explicit description of $Aut(W)$ as a semidirect product of ``tame'' subgroups of $Aut(W)$. We describe these two subgroups.
	Given a right-angled Coxeter group $W$ with Coxeter graph $\Gamma = (S, E)$, let $F(\Gamma)$ be the collection of the $S$-spheric subgroups of $W$, i.e. the $S$-parabolic subgroups $W_{S'}$ of $W$, with $S'$ a finite clique of $\Gamma_{(W, S)}$ (i.e. $m_{(W, S)}(s, s') \in \{ 1, 2 \}$). Let then $Aut(W, F(\Gamma))$ be the subgroup of $Aut(W)$ which stabilizes $F(\Gamma)$, and $Spe(W)$ the subgroup of $Aut(W)$ which stabilizes the conjugacy class of every $s \in S$.
		
	\begin{theorem}[Tits \cite{tits}] Let $W$ be a right-angled Coxeter group. Then $$Aut(W) = Spe(W) \rtimes Aut(W, F(\Gamma)).$$
\end{theorem}

	Evidently, 
$$Inn(W) \subseteq Spe(W) \text{ and } Aut(\Gamma) \subseteq Aut(W, F(\Gamma)),$$ 
where $Aut(\Gamma)$ denotes the automorphism group of the graph $\Gamma$, which is naturally thought as a subgroup of $Aut(W)$, since every automorphism of $\Gamma$ extends canonically to an automorphism of $Aut(W)$. The next proposition shows the connection between $Inn(W)$ and $Aut(\Gamma)$, and the strong rigidity of $W$.

\begin{proposition}\label{castella_remark} Let $W$ be a right-angled Coxeter group. Then
\begin{equation}W \text{ is strongly rigid } \Leftrightarrow Inn(W) = Spe(W) \text{ and } Aut(\Gamma) = Aut(W, F(\Gamma)).
\end{equation} 
\end{proposition}

\begin{proof} \cite[Remark 5(b)]{castella}.
\end{proof}

	We are then interested in criteria which ensure that the two containments in (2.2) are equalities. The next theorem recapitulates what is known on the subject. We first introduce some definitions which will be useful for the statement of the theorem.
	
	\begin{definition} Let $\Gamma = (V, E)$ be a graph. 
	\begin{enumerate}[(1)]
	\item For $v \in \Gamma$, we let $N(v) = \left\{ v' \in \Gamma : v E v' \right\}$ and $st(v) = N(v) \cup \left\{ v \right\}$.
	\item We say that $\Gamma$ is {\em star-connected} if for every $v \in \Gamma$ we have that $\Gamma - st(v)$ is connected. 
	\item We say that $\Gamma$ has the {\em star property} if for every $v \neq v' \in \Gamma$ we have that $st(v) \not\subseteq st(v')$.
\end{enumerate}
\end{definition}

	\begin{theorem}\label{recap_th} Let $W$ be a right-angled Coxeter group.
\begin{enumerate}[(a)]
\item $ Aut(W, F(\Gamma)) = Aut(\Gamma)$ if and only if $\Gamma_{W}$ has the star property (cfr. \cite{castella}, Proposition 7).
\item If $W$ is of finite rank, then $Spe(W) = Inn(W)$ if and only if $\Gamma_W$ is star-connected (cfr. \cite{muhlherr}, corollary to the main theorem).
\item If $W$ is of arbitrary rank, then the star-connectedness of $\Gamma_W$ is a necessary but not sufficient condition for $Spe(W) = Inn(W)$.
\end{enumerate}
\end{theorem}

	\begin{proof} (c) For the necessity of the condition see \cite[Proposition 5]{tits}. The non-sufficiency of the condition is claimed in \cite{tits}, in the final remark of Section 3, but the exhibited map is not surjective. We thus show the non-sufficiency of the condition. Let $\Gamma =  \bigcup_{i < \omega} \Gamma_i$ be a countably infinite star-connected graph such that for each $i < \omega$ we have that $\Gamma_i$ is finite and there exists $a_i \neq b_i \in \Gamma_i - \Gamma_{i-1}$ such that $a_i$ is not adjacent to $b_i$ and $a_i$ is adjacent to every element in $\Gamma_{i-1}$. Such a $\Gamma = \bigcup_{i < \omega} \Gamma_i$ can easily be found, take e.g. the countably infinite random graph. For every $i < \omega$, let $\alpha_i \in Spe(W_{\Gamma_i})$ be such for every $x \in \Gamma_i$ we have $\alpha_i(x) = a_0 \cdots a_i x a_i \cdots a_0$. Then for every $i \leq j < \omega$ we have that $\alpha_j$ restricted to $W_{\Gamma_i}$ equals to $\alpha_i$, and so $\alpha = \bigcup_{i < \omega}\alpha_i \in Spe(W_{\Gamma})$. But obviously $\alpha \not\in Inn(W_{\Gamma})$.
\end{proof}

	Point (c) above was already observed in \cite{tits}, and also noticed in \cite{castella}, where it is also shown that the star property is equivalent to one of the two conditions used in \cite{brady} to characterize strong rigidity in the finite rank case. In the case of right-angled Coxeter groups of arbitrary rank a necessary and sufficient condition on $\Gamma_W$ ensuring $Spe(W) = Inn(W)$ is not known. In the next two theorems, relying on technology from \cite{servatius} and \cite{tits}, we establish two sufficient conditions for $Spe(W) = Inn(W)$. 
	We first need to develop some combinatorics of right-angled Coxeter groups. Let $(W, S)$ be a Coxeter system. Each element $w \in W$ can be written as
a product of generators:
$$w = s_1 s_2 \cdots s_k,$$
with $s_i \in S$. (The identity element $e$ is represented by the empty word.) If $k$ is minimal among all such expressions for $w$, then $k$ is called the length of $w$ (written as  $|w| = k$) and the word $s_1 s_2 \cdots s_k$ is called a {\em normal form}
(or reduced word) for $w$. We denote by $sp(w)$ the set of letters appearing in {\em any} normal form for $w$, and call it the {\em support} of $W$ with respect to the Coxeter basis $S$. This is well-defined, since if $s_1 s_2 \cdots s_k$ and $s'_1 s'_2 \cdots s'_k$ are two normal forms for $w$, then the set of letters appearing in the word $s_1 s_2 \cdots s_k$ equals to the set of letters appearing in $s'_1 s'_2 \cdots s'_k$ (cfr. e.g. \cite[Corollary 1.4.8]{bjorner}). We now describe two ``moves'' which take a word $s_1 s_2 \cdots s_k$ in $(W, S)$ and change it into another word in $(W, S)$ that represents the same elements of $W$ and which is at most as long:
\begin{enumerate}[$(M_1)$]
\item if $s_i = s_{i+1}$ cancel the letters $s_i$ and $s_{i+1}$;
\item if $m(s_i, s_{i+1}) = 2$ exchange $s_i$ and $s_{i+1}$.
\end{enumerate}
	
	\begin{theorem}[Tits \cite{tits_words}]\label{normal_form} Let $(W, S)$ be a right-angled Coxeter system. If $s_1 s_2 \cdots s_n$ and  $s'_1 s'_2 \cdots s'_m$ are two words representing the same element $w \in W$, then $s_1 s_2 \cdots s_n$ and $s'_1 s'_2 \cdots s'_m$ can be reduced to identical normal forms using moves $(M_1)$ and $(M_2)$.
\end{theorem}

	\begin{proposition}\label{bark_prop_2} Let $s_1 \cdots s_n$ be a word in the right-angled Coxeter system $(W,S)$. Then $s_1 \cdots s_n$ is a normal form if and only if for every $1 \leq i < j \leq k$ with $s_i = s_j$, there exists $i < l < j$ such that $s_l \not\in st(s_i)$.
\end{proposition}

	\begin{proof} See e.g. \cite[Lemma 21]{bark}.	
\end{proof}

	\begin{proposition}\label{bark_prop} Let $s_1 \cdots s_n$ be a word in the right-angled Coxeter system $(W,S)$, and suppose that $s_i$ and $s_j$ can be brought next to each other using $(M_2)$ moves in order to use the move $(M_1)$ to shorten the word $s_1 \cdots s_n$. Then $s_i$ and $s_j$ can be brought together using only moves each of which involves either $s_i$ or $s_j$.
\end{proposition}

	\begin{proof} See e.g. \cite[Lemma 18]{bark} (where it is proved more).	
\end{proof}

	We now prove some facts about reflections (see definition below) in right-angled Coxeter groups. In this section we will only use Corollary \ref{reflections_prop}, but the rest will be crucial in what follows. Specifically, Lemma \ref{subword_lemma} will be the main ingredient in the proof of Theorem \ref{almost_AEC_th}. 

	\begin{definition} Let $(W, S)$ be a Coxeter system. We define the set of {\em reflections} of $(W, S)$ to be the set $R(W, S) = \{ wsw^{-1} : s \in S, w \in W \}$.
\end{definition}

	\begin{lemma}\label{lemma_for_subword_lemma} Let $(W, S)$ be a right-angled Coxeter system, $wsw^{-1} \in R(W,S)$ and $a_1 \cdots a_k$ a normal form for $w$. If $a_1 \cdots a_k s a_k \cdots a_1$ is not a normal form for $wsw^{-1}$, then there exists $1 \leq i \leq k$ such that:
\begin{enumerate}[(a)]
	\item $wsw = a_1 \cdots a_{i-1} a_{i+1} \cdots a_k s a_k \cdots a_{i+1}a_{i-1} \cdots a_1$;
	\item $a_i$ commmutes with $a_j$ for every $i < j \leq k$;
	\item $a_i$ commmutes with $s$;
	\item $a_1 \cdots a_{i-1} a_{i+1} \cdots a_k$ is a normal form.
\end{enumerate}
\end{lemma}

	\begin{proof} If $a_1 \cdots a_k s a_k \cdots a_1 = b_1 \cdots b_{2k+1}$ is not a normal form for $wsw^{-1}$, then because of Theorem \ref{normal_form} and the fact that $a_k \cdots a_1$ is normal, it must be the case that in any reduction of $a_1 \cdots a_k s a_k \cdots a_1$ to a normal form at some point we use the move $(M_1)$ for the pair $(b_x,b_y)$, where $x < y$ and either
	\begin{enumerate}[(i)]
	\item $b_x = b_i$ for $i \leq k$ and $b_y = b_{k+1} = s$, or
	\item $b_y = b_i$ for $k+2 \leq i \leq 2k+1$ and $b_x = b_{k+1} = s$, or
	\item $b_x = b_i$ for $i \leq k$ and $b_y = b_j$ for $k+2 \leq j \leq 2k+1$.
	\end{enumerate}
Furthermore, because of Proposition \ref{bark_prop}, we can assume that in this reduction we only use moves that involve either $b_x$ or $b_y$. Now, if we are in case (iii), then it is clear that $i$ is as wanted. In fact it must be the case that $j = (2k+1)-(i-1)$, otherwise $a_1 \cdots a_k$ is not normal, and so we satisfy condition (a) because of our assumption that we use only moves that involve either $b_x$ or $b_y$. Furthermore, conditions (b) and (c) are satisfied because of Proposition \ref{bark_prop_2}. Finally, it is easy to see that also (d) is satisfied, because otherwise $a_1 \cdots a_k$ is not normal. Case (i) and (ii) are symmetric, and so it suffices to analyse case (i). But this is essentially as in case (iii), since after deleting the pair $(b_x,b_y)$ we can move $b_{(2k+1)-(i-1)} = s$ where $b_y = b_{k+1} = s$ was, i.e. in the middle of the word.
\end{proof}

	\begin{lemma}\label{subword_lemma} Let $(W, S)$ be a right-angled Coxeter system, $T \subseteq S$ and $wsw^{-1} \in R(W, S) \cap W_T$. Let $a_1 \cdots a_k$ be a normal form for $w$, and $a_{q_{1}} \cdots a_{q_n}$ be the subword of $a_1 \cdots a_k$ obtained by deleting all the occurrences of letters in $S - T$. Then
	$$wsw^{-1} = a_{q_{1}} \cdots a_{q_n} s a_{q_n} \cdots a_{q_{1}}.$$
\end{lemma}

	\begin{proof} Iterating Lemma \ref{lemma_for_subword_lemma}, we get $l \leq k$ and a sequence of words $(w_i)_{i \leq l}$ such that:
	\begin{enumerate}[(i)]
	\item $w_0 = a_1 \cdots a_k$;
	\item for every $i < l$, the word $w_{i+1}$ is a subword of $w_i$ of length $|w_i| - 1$;
	\item for every $i \leq l$, the word $w_i$ is normal;
	\item for every $i \leq l$, $w_isw_i^{-1} = wsw$;
	\item $w_lsw_l^{-1}$ is normal (and so $sp(w_l), sp(s) \subseteq T$);
	\item $w_l$ is a subword of $a_{q_{1}} \cdots a_{q_n}$.
\end{enumerate}
	For $i < l$, let $a_i$ be the letter witnessing that $w_{l-i}$ is a subword of $w_{l-(i+1)}$ of length $|w_{l-(i+1)}| - 1$, and consider the sequence $((a_i, a_i))_{i < l}$. Then, because of conditions (b) and (c) of Lemma \ref{lemma_for_subword_lemma}, for every $X = \{ i_1 < \cdots < i_{m} \} \subseteq l$, the pairs $((a_i, a_i))_{i \in X}$ can be put back into the word $w_lsw_l^{-1}$ following the order $(a_{i_1}, a_{i_1}) < \cdots < (a_{i_{m}}, a_{i_{m}})$. This suffices, since $w_l$ is a subword of $a_{q_{1}} \cdots a_{q_n}$.
\end{proof}

	The following corollary is immediate from Lemma \ref{subword_lemma}. This is fact is known for any Coxeter group, see e.g. \cite[Corollary 1.4]{gal}.

	\begin{corollary}\label{reflections_prop} Let $(W, S)$ be a Coxeter system and $T \subseteq S$. Then
$$R(W, S) \cap W_T = R(W_T, T).$$
\end{corollary}

	We also need an explicit description of centralizers of Coxeter generators.

	\begin{lemma}[Tits \cite{tits}]\label{centralizers} Let $W$ be a right-angled Coxeter group and $v \in \Gamma_W$. Then the centralizer $C_W(v)$ of $v$ in $W$ is the parabolic subgroup $W_{st(v)}$.
\end{lemma}

\begin{proof} \cite[Corollary 3]{tits}.
\end{proof}
	
	We now go back to the main theme of this section, i.e. strong rigidity. To this end, we need two lemmas. These lemmas are essentially Theorem 3 and Lemma 4 of \cite{servatius} proved in the context of Coxeter groups (\cite{servatius} proves this fact for Artin groups (a.k.a graph groups)).

	\begin{lemma}\label{Servatius_lemma} Let $W$ be a right-angled Coxeter group, $\alpha \in Spe(W)$, $v \in \Gamma_W$ and $Y$ a connected component of $\Gamma_W - st(v)$. Then if $v \in sp(\alpha(y))$ for some $y \in Y$, then $v \in sp(\alpha(x))$ for every $x \in Y$.
\end{lemma}

	\begin{proof} We show that $v \in sp(\alpha(x))$ for any $x$ adjacent to $y$ and not adjacent to $v$, the result follows by the connectedness of $Y$. Now, $\alpha(y) = wyw^{-1}$ for some $w \in W$, because $\alpha \in Spe(W)$, and $sp(wyw^{-1}) \subseteq sp(w) \cup sp(y)$ (cfr. Theorem \ref{normal_form}). By hypothesis $v \in sp(\alpha(y))$, and evidently $v \not\in sp(y) = \{ y \}$, thus $v \in sp(w)$. Consider now $\alpha(x)$. As for $\alpha(y)$, there exists $p \in W$ such that $\alpha(x) = pxp^{-1}$. By the choice of $x$, the element $y$ commutes with $x$, and so $\alpha(y)$ commutes with $\alpha(x)$. That is, $\alpha(x) \in C_W(\alpha(y))$. By Lemma \ref{centralizers}
	$$C_W(wyw^{-1}) = wC_W(y)w^{-1} = wW_{st(y)}w^{-1},$$
and so $\alpha(x) \in wW_{st(y)}w^{-1}$, i.e. $\alpha(x) = wy'w^{-1}$ for some $y' \in st(y)$. Furthermore, being $\alpha(x)$ conjugate to $x$, we have $x \in sp(\alpha(x)) = sp(wy'w^{-1})$. We distinguish two cases.
\newline {\bf Case 1.} $x \in sp(y')$. If this is the case, then $v \in sp(wy'w^{-1})$, because $x$ is not adjacent to $v$ (cfr. Theorem \ref{normal_form}).
\newline {\bf Case 2.} $x \not\in sp(y')$. We show that this case is not possible. If $x \not\in sp(y')$, then $x \in sp(w) -sp(y')$. Thus, for any normal form $w_1 \cdots w_k$ and $y'_1 \cdots y'_m$ for $w$ and $y'$, respectively, we have that $x$ occurs an even number of times in $$w_1 \cdots w_ky'_1 \cdots y'_mw_k \cdots w_1.$$
Hence, $x$ occurs an even number of times also in $p_1 \cdots p_l x p_l \cdots p_1$, for $p_1 \cdots p_l$ a normal form for $p$ (see e.g. \cite[pg. 14]{meyers}), but this is obviously absurd.
\end{proof}

	\begin{lemma}\label{first_strong_rig_lemma} Let $W$ be a right-angled Coxeter group such that $\Gamma_W$ satisfies the following conditions:
\begin{enumerate}[(a)]
\item $\Gamma_W$ is star-connected;
\item $\Gamma_W$ is triangle-free;
\item $\Gamma_W$ contains a copy of $P_4$ (the path of length $4$) as a subgraph (not necessarily induced).
\end{enumerate}
Then for every $\alpha \in Spe(W)$ there exists $w \in W$ such that $w \alpha w^{-1}$ fixes $P_4$ pointwise.
\end{lemma}

	\begin{proof} Let $P_4 = aEbEcEd$ and $\alpha \in Spe(W)$. Then $\alpha(a) = pap^{-1}$ and so conjugating $\alpha$ by $p^{-1}$ we get $\alpha_1 \in Spe(W)$ such that $\alpha_1(a) = a$. Now, $a$ and $b$ commute and so we have $\alpha_1(b) = qbq^{-1}$ with $sp(q) \subseteq N(a)$ (cfr. Lemma \ref{centralizers}). Thus, conjugating $\alpha_1$ by $q^{-1}$ we get $\alpha_2 \in Spe(W)$ such that $\alpha_2(a) = a$ and $\alpha_2(b) = b$. Similarly, $b$ and $c$ commute and so we have $\alpha_2(c) = rcr^{-1}$ with $sp(r) \subseteq N(b)$. Let $x \in N(b) - \{ a, c\}$, then by the triangle-freeness of $\Gamma_W$, $x$ is adjacent neither to $a$ nor to $c$, and so $a, c \in \Gamma_W - st(x)$. By the star-connectedness of $\Gamma_W$, $a$ and $c$ are connected in $\Gamma_W - st(x)$, and so given that $x \not\in sp(\alpha_2(a)) = sp(a) = \{ a \}$, by Lemma \ref{Servatius_lemma} we have $x \not\in sp(\alpha_2(c))$. Hence $sp(\alpha_2(c)) \subseteq \{ a, c \}$. Then $\langle \alpha_2(a) = a, \alpha_2(c) \rangle_W \subseteq \langle a, c \rangle_W$. On the other hand, $\alpha_2^{-1} \in Spe(W)$, $\alpha_2^{-1}(a) = a$ and $\alpha_2^{-1}(b) = b$, and so the same argument used for $\alpha_2$ shows that $sp(\alpha_2^{-1}(c)) \subseteq \{ a, c \}$. Thus, $\alpha_2^{-1}(c) \in \langle a, c \rangle_W$, from which it follows that
	$$c \in \alpha_2(\langle a, c \rangle_W) = \langle \alpha_2(a) = a, \alpha_2(c) \rangle_W,$$
i.e. $\langle a, c \rangle_W \subseteq \langle a, \alpha_2(c) \rangle_W$. Hence,
	$$\langle a, c \rangle_W = \langle a, \alpha_2(c) \rangle_W.$$
That is, $\alpha_2$ restricted to $\langle a, c \rangle_W = W_{\{ a, c \}} \in Aut(W_{\{ a, c \}})$. Furthermore, because of Corollary \ref{reflections_prop} we see that $\alpha_2 \in Spe(W_{\{ a, c \}})$. Also, $(\{ a, c \}, E) = (\{ a, c \}, \emptyset)$ is star-connected, and so by Theorem \ref{recap_th}(b) we have $\alpha_2 \in Inn(W_{\{ a, c \}})$. But then obviously it must be the case that $\alpha_2(c)$ is either $c$ or $aca$, because otherwise $\alpha_2(a) \neq a$. It follows that $sp(r) \subseteq \{ a \}$, and so conjugating $\alpha_2$ by $r^{-1}$ we get $\alpha_3 \in Spe(W)$ such that $\alpha_3(a) = a$, $\alpha_3(b) = b$ and $\alpha_3(c) = c$. Using the same argument for $\alpha_3(d) = tdt^{-1}$, we see that $sp(t) \subseteq \{ b \}$, and so conjugating $\alpha_3$ by $t^{-1}$ we get $\alpha_4 \in Spe(W)$ such that $\alpha_4(a) = a$, $\alpha_4(b) = b$, $\alpha_4(c) = c$ and $\alpha_4(d) = d$.
\end{proof}

	We now arrive at the first sufficient condition for $Spe(W) = Inn(W)$. This theorem takes inspiration from \cite[Theorem 6]{servatius}, where his use of Theorem 3 and Lemma 4 is replaced by our Lemmas \ref{Servatius_lemma} and \ref{first_strong_rig_lemma}.

	\begin{theorem}\label{first_strong_rig_th} Let $W$ be a right-angled Coxeter group such that $\Gamma_W$ satisfies the following conditions:
\begin{enumerate}[(a)]
\item $\Gamma_W$ is star-connected;
\item $\Gamma_W$ is triangle-free;
\item $\Gamma_W$ contains $P_4$ as a subgraph.
\end{enumerate}
Then $Spe(W) = Inn(W)$.
\end{theorem}

	\begin{proof} Let $\alpha \in Spe(W)$, then by Lemma \ref{first_strong_rig_lemma} there exists $w \in W$ such that $w \alpha w^{-1}$ fixes $P_4 = aEbEcEd$ pointwise. We show that $\alpha_{1} = w \alpha w^{-1}$ is the identity $id_W$ on $W$. This of course suffices, since then 
	$$\alpha = w^{-1}w \alpha w^{-1}w = w^{-1}id_Ww = Inn(w^{-1}),$$
where, for $x \in W$, $Inn(x)$ denotes the inner automorphism determined by $x$. To this end, let $y \not\in P_4$ and suppose that $\alpha_1(y) \neq g y g^{-1}$. Then there is $v \in sp(g)$ such that $v \neq y$ and $v$ is not adjacent to $y$. By the triangle-freeness of $\Gamma_W$ there exists $e \in \{ a, b, c, d \} - \{ v \}$ such that $e$ is not adjacent to $v$. It follows that $\Gamma - st(v)$ contains $y$ and $e$. Furthermore, $v \in sp(\alpha_1(y))$ and so by Lemma \ref{Servatius_lemma} we have 
	$$v \in sp(\alpha_1(e)) = sp(e) = \{e\},$$
which is a contradiction. Thus, we must have $\alpha_1(y) = y$. It follows that $\alpha_1 = id_W$.
\end{proof}

	\begin{corollary}\label{first_strong_rig_cor} Let $W$ be as in Theorem \ref{first_strong_rig_th} and suppose that in addition $\Gamma_W$ has the star property. Then $W$ is strongly rigid.
\end{corollary}

	\begin{proof} Immediate from Proposition \ref{castella_remark}, Theorem \ref{recap_th} and Theorem \ref{first_strong_rig_th}.
\end{proof}

	Finally, we arrive at the second sufficient condition for $Spe(W) = Inn(W)$. This theorem takes inspiration from \cite[Proposition 6]{tits}, although the setting of the reference is quite different from the one in the theorem.

	\begin{theorem}\label{second_strong_rig_th} Let $W$ be a right-angled Coxeter group such that $\Gamma_W$ satisfies the following conditions:
\begin{enumerate}[(a)]
\item $\Gamma_W$ is star-connected;
\end{enumerate}
and either $\Gamma_W$ is finite or there exists $s, s' \in \Gamma_W$ such that:
\begin{enumerate}[(a)]\setcounter{enumi}{1}
\item  $st(s) \cup st(s')$ is finite and star-connected (as an induced subgraph);
\item for every $v \in \Gamma_W$ there exists $a \in st(s) \cup st(s')$ such that $a \neq v$ and $a$ is not adjacent to $v$.
\end{enumerate}
Then $Spe(W) = Inn(W)$.
\end{theorem}

\begin{proof} If $\Gamma_W$ is finite, then we know that star-connectedness suffices for $Spe(W) = Inn(W)$. Suppose then that $\Gamma_W$ is infinite (and so conditions (b) and (c) hold). Let $s, s'$ be as in the statement of the theorem and $\alpha \in Spe(W)$. We will show that there exists $w \in W$ such that $w \alpha w^{-1}$ is the identity on $W$. By assumption $\alpha(s) = psp^{-1}$ and so conjugating $\alpha$ by $p^{-1}$ we get $\alpha_1 \in Spe(W)$ such that  $\alpha_1(s) = s$.
Now, $s$ and $s'$ commute and so we have $\alpha_1(s') = qs'q^{-1}$ with $sp(q) \subseteq N(a)$ (cfr. Lemma \ref{centralizers}). Thus, conjugating $\alpha_1$ by $q^{-1}$ we get $\alpha_2 \in Spe(W)$ such that $\alpha_2(s) = s$ and $\alpha_2(s') = s'$.
Given that $C_W(s) = W_{st(s)}$ and $C_W(s') = W_{st(s')}$(cfr. Lemma \ref{centralizers}) we must have that $\alpha_2$ fixes $W_{st(s) \cup st(s')}$ setwise, i.e. $\alpha_2$ restricted to $W_{st(s) \cup st(s')}$ is in $Aut(W_{st(s) \cup st(s')})$. Furthermore, because of Corollary \ref{reflections_prop} we see that $\alpha_2 \in Spe(W_{st(s) \cup st(s')})$. Also, by assumption $st(s) \cup st(s')$ is finite and star-connected, and so we have $\alpha_2 \in Inn(W_{st(s) \cup st(s')})$ (cfr. Theorem \ref{recap_th}(b)). Thus, composing $\alpha_2$ with an inner automorphism, we get $\alpha_3 \in Spe(W)$ which fixes $W_{st(s) \cup st(s')}$ pointwise. We show that $\alpha_3$ fixes every element of $\Gamma_W$. To this end, let $y \not\in st(s) \cup st(s')$ and suppose that $\alpha_3(y) = g y g^{-1}$ is not fixed. Then there is $v \in sp(g)$ such that $v \neq y$ and $v$ is not adjacent to $y$. Notice that because of (c) there exists $a \neq v \in st(s) \cup st(s')$ such that $v$ is not adjacent to $a$. It follows that $\Gamma - st(v)$ contains $y$ and $a$. Furthermore, $v \in sp(\alpha_3(y))$, and so by Lemma \ref{Servatius_lemma} we have
	$$v \in sp(\alpha_3(a)) = sp(a) = \{a\},$$ which is a contradiction. Thus, we must have $\alpha_3(y) = y$. It follows that $\alpha_3 = id_W$.
\end{proof}

	\begin{corollary}\label{second_strong_rig_cor} Let $W$ be as in Theorem \ref{second_strong_rig_th} and suppose that in addition $\Gamma_W$ has the star property. Then $W$ is strongly rigid.
\end{corollary}

	\begin{proof} Immediate from Proposition \ref{castella_remark}, Theorem \ref{recap_th} and Theorem \ref{second_strong_rig_th}.
\end{proof}

	We will refer to groups satisfying the conditions of Corollary \ref{second_strong_rig_cor} as {\em centered} right-angled Coxeter groups (centered because of the $s$ and $s'$). 
	


\section{Random Right-Angled Coxeter Groups}

	Let $T_{rg}$ be the first-order theory of random graphs, and $T_{racg}$ be $Th(A)$ for $A$ {\em any} right-angled Coxeter group such that $\Gamma_A \models T_{rg}$. This does not depend on $A$, since for every right-angled Coxeter groups $B$ and $C$ such that $\Gamma_B, \Gamma_C \models T_{rg}$ the two groups $B$ and $C$ are elementary equivalent. This can be seen using e.g. the Ehrenfeucht-Fra\"iss\'e game $EF_{\omega}(B, C)$ of length $\omega$ (this definitely suffices, since it shows that $B$ and $C$ are elementary equivalent in the infinitary logic $L_{\infty, \omega}$). We sketch the idea. If in the game $EF_{\omega}(B, C)$ Player I plays an element $b_0 \in B$ with normal form $s^0_1 \cdots s^0_n$, then Player II plays the element $c_0 = t^0_1 \cdots t^0_n$, for $t^0_1 \cdots t^0_n$ the answer of Player II to the move $s^0_1 \cdots s^0_n$ of Player I in the game $EF_{\omega}(\Gamma_B, \Gamma_C)$, in which, as well-known, Player II has a winning strategy, since $\Gamma_B, \Gamma_C \models T_{rg}$. (Notice that in a game of length $\omega$ playing elements or tuples does not matter.) The other moves are played in the same fashion. 
	
	We now fix a cardinal $\kappa > \omega$ such that $\kappa^{<\kappa} = \kappa$ and code models $A$ of cardinalilty $\kappa$ in a universal countable language $L^*$ (countably many relation symbols for any arity) as elements $\eta(A)$ of $2^{\kappa}$ in the usual fashion (see e.g. \cite{Fr&Hy&Ku}). Given a complete first-order theory $T$ in the language $L^*$, we define the isomorphism relation $\cong_T$ on $2^{\kappa} \times 2^{\kappa}$ as the relation 
	$$ \{ (\eta(A), \eta(B)) \in 2^{\kappa} \times 2^{\kappa} : A, B \models T, A \cong B \} 		  \cup \{ (\eta(A), \eta(B)) \in 2^{\kappa} \times 2^{\kappa} : A, B \not\models T \}.$$ 
Finally, given two complete first-order theories $T_0$ and $T_1$ in the language $L^*$ we can say that the isomorphism relation of $T_0$ reduces to the isomorphism relation of $T_1$, denoted as $\cong_{T_0} \; \leq_B \; \cong_{T_1}$, if the relation $\cong_{T_0}$ is Borel reducible to $\cong_{T_1}$ in the usual sense of generalized descriptive set theory (cfr. e.g. \cite{Fr&Hy&Ku}).
Clearly any (complete) countable first-order theory can be thought canonically as a (complete) theory in the language $L^*$ (in particular $T_{rg}$ and $T_{racg}$ can be thought so). We denote by $\cong_{RACG}$ the isomorphism relation $\cong_{T_{racg}}$. Given a graph $\Gamma =( V, E)$ and $X \subseteq V$ we say that $V$ is a clique (resp. an indepenent set) if for every $x \neq y \in X$ we have $xEy$ (resp. $x$ is not adjacent to $y$).

	\begin{theorem}\label{red_to_random_graphs} For any countable complete first-order theory $T$, $$\cong_T \; \leq_B \; \cong_{T_{rg}}.$$
\end{theorem}

	\begin{proof} This is folklore, we sketch a proof for completeness of exposition. As well-known, it suffices to do the following: for every graph $\Gamma$ of power $\kappa$ we define a random graph $R_{\Gamma}$ of power $\kappa$ such that $\Gamma \cong \Gamma'$ iff $R_{\Gamma} \cong R_{\Gamma'}$. We do this. Let $\Gamma = (V, E)$ be a graph of power $\kappa$ with $V \cap E = \emptyset$ (without loss of generality). Define a graph
$R^0_{\Gamma}$ on $V \cup E$ by letting $a$ and $b$ be adjacent to $\{a, b\}$, for every $\{ a, b\} \in E$.  Now, for every $a \in V$ add a clique $K_a$ of size $\omega_1$ such that $a$ is adjacent to co-countably many $x \in K_a$, i.e. $K_a - N(a)$ has size $\omega$. Similarly, for every $\{a,b\} \in E$ add a clique of size $\omega_1$ such that $K_{a,b} \cap N(\{a,b\})$ and $K_{a,b} - N(\{a,b\})$ have both size $\omega_1$. Let $R^1_{\Gamma}(0)$ be the resulting graph, and define $R^2_{\Gamma}(i+1)$ by closing $R^2_{\Gamma}(i)$ under the following condition: for every finite $X$ there exists $a_X$ such that $N(a_X) = X$. Then $\bigcup_{i < \omega} R^2_{\Gamma}(i) = R_{\Gamma} \models T_{rg}$ is as wanted.
\end{proof}
	
	\begin{theorem} For any countable complete first-order theory $T$, $$\cong_T \; \leq_B \; \cong_{RACG}.$$
\end{theorem}

	\begin{proof} Because of Theorem \ref{red_to_random_graphs}, it suffices to show that $\cong_T \; \leq_B \; \cong_{RACG}$ for $T = T_{rg}$ the theory of random graphs. But this is immediate since we can define $F: 2^{\kappa} \rightarrow 2^{\kappa}$ by setting
	$$F(\eta(\Gamma)) = \begin{cases} 
	 \eta(A) 	      \;\;\; \text{ if }  \Gamma \not\models T_{rg} \\
	 \eta(A_{\Gamma}) \; \text{ if }  \Gamma \models T_{rg},
\end{cases}$$
where in the first clause $A$ denotes any fixed right-angled Coxeter group $A$ such that $A \not\models T_{racg}$, and in the second clause $A_{\Gamma}$ is the right-angled Coxeter group of type $\Gamma$. The function $F$ is evidently Borel.
\end{proof}

	The following result shows that no non-trivial class of right-angled Coxeter groups can be treated from the perspective of first-order model theory. This motivates our use of abstract elementary classes.

	\begin{theorem} Let $\mathbf{K}$ be a class of right-angled Coxeter groups such that there exists $A \in \mathbf{K}$ with $\Gamma_A$ containing two non-adjacent vertices $a$ and $b$. Then $\mathbf{K}$ is not first-order axiomatizable.
\end{theorem}

	\begin{proof} Let $A$, $a$ and $b$ be as in the statement of the theorem. Then for every positive integer $n$ the element $c_n = (ab)^n \in A$ is divisible by $n$. It follows that in the ultrapower $\prod_{i < \omega} A_i/U$ ($U$ non-principal ultrafilter) there exists a divisible element $c$ (i.e. an element divisible by every positive integer $n$), but a Coxeter group can not contain such an element $c$. Thus, $\prod_{i < \omega} A_i/U \not\in \mathbf{K}$ (and so $\mathbf{K}$ is not first-order).
\end{proof}

\section{Abstract Elementary Classes}

	In this section we introduce the basics of abstract elementary classes (see e.g. \cite{shelah_abstr_ele_cla} and \cite{jarden}). This machinery will be used in later sections in order to study various classes of right-angled Coxeter groups. As usual in this context, type means Galois type (cf. e.g. \cite[Definition 8.10]{baldwin}). Given a class $\mathbf{K}$ of structures in the vocabulary $L$, we denote by $\leq$ the $L$-submodel relation on structures in $\mathbf{K}$.

\begin{definition}[Abstract Elementary Class \cite{shelah_abstr_ele_cla}]\label{def_AEC}  Let $\mathbf{K}$ be a class of structures in the vocabulary $L$. We say that $(\mathbf{K}, \preccurlyeq)$ is an {\em abstract elementary class} ($\mathrm{AEC}$) if the following conditions are satisfied.
		\begin{enumerate}[(1)]
			\item $\mathbf{K}$ and $\preccurlyeq$ are closed under isomorphism.
			\item If ${A} \preccurlyeq {B}$, then ${A}$ is a substructure of ${B}$ (${A} \leq {B}$).
			\item The relation $\preccurlyeq$ is a partial order on $\mathbf{K}$.
			\item If $({A}_i)_{i < \delta}$ is an increasing continuous $\preccurlyeq$-chain, then:
			\begin{enumerate}[({4.}1)]
				\item $\bigcup_{i < \delta} {A}_i \in \mathbf{K}$;
				\item for each $j < \delta$, ${A}_j \preccurlyeq \bigcup_{i < \delta} {A}_i$;
				\item if each ${A}_j \preccurlyeq {B}$, then $\bigcup_{i < \delta} {A}_i \preccurlyeq {B}$ \; (Smoothness Axiom).
	\end{enumerate}
			\item If ${A}, {B}, {C} \in \mathbf{K}$, ${A} \preccurlyeq {C}$, ${B} \preccurlyeq {C}$ and ${A} \leq {B}$, then ${A} \preccurlyeq {B}$ \; (Coherence Axiom).
			\item There is a L\"owenheim-Skolem number $\mathrm{LS}(\mathbf{K}, \preccurlyeq)$ such that if ${A} \in \mathbf{K}$ and $B \subseteq A$, then there is ${C} \in \mathbf{K}$ such that $B \subseteq C$, ${C} \preccurlyeq {A}$ and $|C| \leq |B| + |L| + \mathrm{LS}(\mathbf{K}, \preccurlyeq)$ \; (Existence of LS-number).	
\end{enumerate}		
\end{definition}

	\begin{definition} If ${A}, {B} \in \mathbf{K}$ and $f: {A} \rightarrow {B}$ is an embedding such that $f({A}) \preccurlyeq {B}$, then we say that $f$ is a $\preccurlyeq$-embedding.
	
\end{definition}
	
	Let $\lambda$ be a cardinal. We let $\mathbf{K}_{\lambda} = \left\{ {A} \in \mathbf{K} \; | \; |A| = \lambda \right\}$.

	\begin{definition}\label{def_AP} Let $(\mathbf{K}, \preccurlyeq)$ be an $\mathrm{AEC}$. 
		\begin{enumerate}[(1)]
			\item We say that $(\mathbf{K}, \preccurlyeq)$ has the {\em amalgamation property} $(\mathrm{AP})$ if for any ${A}, {B}_0, {B}_1 \in \mathbf{K}$ with ${A} \preccurlyeq {B}_i$, for $i < 2$, there are ${C} \in \mathbf{K}$ and $\preccurlyeq$-embeddings $f_i: {B}_i \rightarrow {C}$, for $i < 2$, such that $f_0 \restriction A = f_1 \restriction A$.
			\item We say that $(\mathbf{K}, \preccurlyeq)$ has the {\em joint embedding property} $(\mathrm{JEP})$ if for any ${B}_0, {B}_1 \in \mathbf{K}$ there are ${C} \in \mathbf{K}$ and $\preccurlyeq$-embeddings $f_i: {B}_i \rightarrow {C}$, for $i < 2$.
			\item We say that $(\mathbf{K}, \preccurlyeq)$ has {\em arbitrarily large models} $(\mathrm{ALM})$ if for every $\lambda \geq \!\mathrm{LS}(\mathbf{K}, \!\preccurlyeq)$, $\mathbf{K}_{\lambda} \neq \emptyset$.
\end{enumerate}
\end{definition}

	As well-known, given an $\mathrm{AEC}$, say $(\mathbf{K}, \preccurlyeq)$, with $\mathrm{AP}$, $\mathrm{JEP}$ and $\mathrm{ALM}$, we can construct a monster model $\mathfrak{M} = \mathfrak{M}(\mathbf{K}, \preccurlyeq)$ for $(\mathbf{K}, \preccurlyeq)$, i.e. a $\kappa$-model homogeneous and $\kappa$-universal (for $\kappa$ large enough) structure in $\mathbf{K}$. We say that a subset $A$ of $\mathfrak{M}$ is bounded if its cardinality is smaller than $\kappa$. Given bounded $A \subseteq \mathfrak{M}$ and $n < \omega$, we denote by $S_n(A)$ the set of Galois types\footnote{For a definition of Galois type see e.g. \cite[beginning of Section 4]{kueker}.} over $A$ of length $n$, and by $S(A)$ the set $\bigcup_{n < \omega} S_n(A)$.
	
		\begin{definition} Let $(\mathbf{K}, \preccurlyeq)$ be an $\mathrm{AEC}$ with $\mathrm{AP}$, $\mathrm{JEP}$ and $\mathrm{ALM}$. We say that $(\mathbf{K}, \preccurlyeq)$ has the {\em independence property} if there exists finite $A \subseteq \mathfrak{M}$ and $P \subseteq S(A)$ such that for every ordinal $\alpha < |\mathfrak{M}|$ there exist $(a_i)_{i < \alpha} \in \mathfrak{M}$ such that for every $X \subseteq \alpha$ there exists $b_X \in \mathfrak{M}$ such that $tp(b_Xa_i/A) \in P$ if and only if $i \in X$.
\end{definition}

	\begin{definition} Let $(\mathbf{K}, \preccurlyeq)$ be an $\mathrm{AEC}$ with $\mathrm{AP}$, $\mathrm{JEP}$ and $\mathrm{ALM}$. We say that $(\mathbf{K}, \preccurlyeq)$ is {\em homogeneous} if for every ordinal $\alpha < |\mathfrak{M}|$ and $(a_i)_{i< \alpha}, (b_i)_{i< \alpha} \in \mathfrak{M}$, if $tp(a_X) = tp(b_X)$ for every $X \subseteq_{fin} \alpha$, then $tp((a_i)_{i< \alpha}) = tp((b_i)_{i< \alpha})$.
\end{definition}

	\begin{definition}[\cite{kesala} and \cite{kueker}] Let $(\mathbf{K}, \preccurlyeq)$ be an $\mathrm{AEC}$. We say that $(\mathbf{K}, \preccurlyeq)$ has {\em finite character} if whenever $A \leq B$ and for every $X \subseteq_{fin} A$ there exists $\preccurlyeq$-embedding $f_X: A \rightarrow B$ such that $f \restriction X = id_X$, then $A \preccurlyeq B$.
\end{definition}

\begin{definition}[\cite{kesala}] Let $(\mathbf{K}, \preccurlyeq)$ be an $\mathrm{AEC}$. We say that $(\mathbf{K}, \preccurlyeq)$ is {\em finitary} if the following are satisfied:
\begin{enumerate}[(1)]
\item $\mathrm{LS}(\mathbf{K}, \preccurlyeq) = \omega$;
\item $(\mathbf{K}, \preccurlyeq)$ has arbitrarily large models;
\item $(\mathbf{K}, \preccurlyeq)$ has the amalgamation property;
\item $(\mathbf{K}, \preccurlyeq)$ has the joint embedding property;
\item $(\mathbf{K}, \preccurlyeq)$ has finite character.
\end{enumerate}
\end{definition}

	\begin{definition} Let $(\mathbf{K}, \preccurlyeq)$ be an $\mathrm{AEC}$ with $\mathrm{AP}$, $\mathrm{JEP}$ and $\mathrm{ALM}$. For $\mathrm{LS}(\mathbf{K} ,\preccurlyeq~) \leq \kappa \leq \lambda$, we say that $(\mathbf{K}, \preccurlyeq)$ is {\em $(\kappa, \lambda)$-tame} if for every $B \in \mathbf{K}$ of power $\lambda$ and $a, b \in \mathfrak{M}^{< \omega}$, if $tp(a/B) \neq tp(b/B)$, then there is $A \preccurlyeq B$ of power $\kappa$ such that $tp(a/A) \neq tp(b/A)$.  We say that $(\mathbf{K}, \preccurlyeq)$ is {\em tame} if it is $(\mathrm{LS}(\mathbf{K}, \preccurlyeq), \lambda)$-tame for every $\lambda \geq \mathrm{LS}(\mathbf{K}, \preccurlyeq)$.
\end{definition}

	As usual, we say that $(\mathbf{K}, \preccurlyeq)$ is {\em uncountably categorical} if for every uncountable cardinal $\kappa$ there exists only one model of power $\kappa$, up to isomorphism. In later sections we will use the following classical result on abstract elementary classes.

	\begin{theorem}[see e.g. {\cite[Theorem 8.21]{baldwin}}]\label{cate_implies_stability} Let $(\mathbf{K}, \preccurlyeq)$ be an $\mathrm{AEC}$ with $\mathrm{AP}$, $\mathrm{JEP}$ and $\mathrm{ALM}$. If $(\mathbf{K}, \preccurlyeq)$ is uncountably categorical, then $(\mathbf{K}, \preccurlyeq)$ is stable\footnote{The notion of stability in this context is the exact analogous of the notion of stability in the classical context of first-order logic, where we replace the notion of type with the notion of Galois type. For an explicit definition see e.g. \cite[Definition 8.20]{baldwin}.} in every infinite cardinality $\lambda \geq \mathrm{LS}(\mathbf{K}, \preccurlyeq)$.
\end{theorem}

	We will also use the following results connecting finitary abstract elementary classes with infinitary logic. Given $\theta \in L_{\infty, \omega}$, we let $Mod(\theta) = \{ A : A \models \theta \}$. 
	
	\begin{theorem}[Kueker {\cite[Theorem 3.10]{kueker}}]\label{kueker1} Let $(\mathbf{K}, \preccurlyeq)$ be a finitary $\mathrm{AEC}$ with countable vocabulary. If $\mathbf{K}$ contains at most $\lambda$ models of cardinality $\lambda$ for some infinite $\lambda$, then $\mathbf{K} = Mod(\theta)$ for some $\theta \in L_{\infty, \omega}$. If in addition $\mathbf{K}$ contains at most $\lambda$ models of cardinality $< \lambda$, then we can find $\theta \in L_{\lambda^+, \omega}$.
\end{theorem}

	\begin{definition} Let $(\mathbf{K}, \preccurlyeq)$ be a finitary $\mathrm{AEC}$ with monster model $\mathfrak{M}$. Let also $a \in \mathfrak{M}^{< \omega}$ and $A \preccurlyeq \mathfrak{M}$. Then
$$ tp_{\omega_{1}, \omega}(a/A) = \{ \phi(x, b) : \phi(x, y) \in L_{\omega_1, \omega}, \; b \in A^{< \omega} \text{ and }  \mathfrak{M} \models  \phi(a, b) \}.$$
\end{definition}

	\begin{theorem}[Kueker {\cite[Remark after Corollary 4.9]{kueker}}]\label{kueker2} Let $(\mathbf{K}, \preccurlyeq)$ be a finitary and tame $\mathrm{AEC}$ with countable vocabulary. Assume also that $(\mathbf{K}, \preccurlyeq)$ is $\omega$-stable. Then for every $A \preccurlyeq \mathfrak{M}$ we have $tp_{\omega_{1}, \omega}(a/A) = tp_{\omega_{1}, \omega}(b/A) \text{ iff } tp(a/A) = tp(b/A)$.
\end{theorem}



\section{Triangle-Free Right-Angled Coxeter Groups}\label{triangle_free}
 
	From now till the end of the paper we denote by $\mathbf{K}$ the class of right-angled Coxeter groups, and by $\preccurlyeq$ the parabolic subgroup relation on $\mathbf{K}$ (cfr. Definition \ref{def_parabolic}), i.e. $A \preccurlyeq B$ if and only if there exists a Coxeter basis $S$ for $B$ such that $A \cap S$ is a Coxeter basis for $A$. Also, we denote by $\leq$  both the subgroup and the induced subgraph relation. Finally, we simply talk of bases instead of Coxeter bases. The next theorem shows that $(\mathbf{K}, \preccurlyeq)$ does not give rise to an abstract elementary class. In the rest of the paper we will see that restricting to particular classes of {\em strongly rigid} right-angled Coxeter groups we {\em do get} abstract elementary classes, and actually finitary ones (and in some cases also tame).
	
	\begin{theorem} The Smoothness Axiom fails for $(\mathbf{K}, \preccurlyeq)$.
\end{theorem}

	\begin{proof} Let $(B, S)$ be the Coxeter system with $S = \{ a_i : i < \omega \} \cup \{ b_i : i < \omega_1 \}$ such that $\{ a_i : i < \omega \}$ is an independent set, $\{ b_i : i < \omega_1 \}$ is a clique, and $a_i$ commutes with $b_j$ iff $j < i$, for every $i < \omega$. For $n < \omega$, let $c_n = a_0 \cdots a_n$, $e_n = c_nb_nc_n^{-1}$ and $A_n = \langle e_i : i < n \rangle_B$. Notice that for every $i \leq j < \omega$ we have $c_jb_ic_j^{-1} = c_ib_ic_i^{-1}$. It follows that for every $m < n < \omega$, we have $A_m \preccurlyeq A_n \preccurlyeq B$, as witnessed by the bases $ \{ e_i : i < m \} \subseteq \{ e_i : i < n \}  \subseteq c_nSc_n^{-1}$. We claim that $\bigcup_{n < \omega} A_n = A \not\preccurlyeq B$. Suppose not, and let $S^*$ be a basis of $A$ that extends to a basis $S'$ of $B$. Let $\alpha \in Aut(B)$ be such that $\alpha(S') = S$. Then $\alpha(S^*) \subseteq \{ b_i : i < \omega_1\}$, and so there exists $x \in S-\alpha(S^*)$ such that $x$ commutes with every element of $\alpha(S^*)$. Let $y = \alpha^{-1}(x)$, then $y$ commutes with every element of $A$. Let $n < \omega$ be such that if $b_i$ or $a_i$ is in the $S$-support of $y$, then either $i \geq \omega$ or $i < n$. Also, let $z = c_n^{-1}yc_n$. Now, $y$ commutes with every element of $A$, and so in particular it commutes with $e_n$. Thus, $z = c_n^{-1}yc_n$ commutes with $c_n^{-1}e_nc_n=b_n$. Now, if for some $i \geq n$, $b_i$ is in the $S$-support of $z$, then also $a_n$ is there and so $z$ does not commutes with $b_n$ (cfr. Lemma \ref{centralizers}). Similarly, for every $i < \omega$, $a_i$ is not in the $S$-support of $z$. Thus, $z \in \langle b_i : i < n \rangle_B$ and so $c_nzc_n^{-1} = y \in \langle c_nb_ic_n^{-1} : i < n \rangle_B = A_n$, which is a contradiction, since $y = \alpha^{-1}(x)$, for $x \in S-\alpha(S^*)$.
\end{proof}
	
	\begin{theorem}\label{almost_AEC_th} Let $\mathbf{K}'_{*}$ be a class of graphs such that $(\mathbf{K}'_{*}, \leq)$ is closed under limits and every $B \in \mathbf{K}_{*} = \left\{ A \in \mathbf{K} : \Gamma_A \in \mathbf{K}'_{*} \right\}$ is strongly rigid. Then $(\mathbf{K}_{*}, \preccurlyeq)$ satisfies conditions (1), (2), (3), (4.1), (4.2) and (5) of Definition \ref{def_AEC}. Furthermore, $\mathrm{LS}(\mathbf{K}_{*}, \preccurlyeq) = \mathrm{LS}(\mathbf{K}'_{*}, \leq)$, and if $(\mathbf{K}'_{*}, \leq)$ has $\mathrm{AP}$, $\mathrm{JEP}$ and $\mathrm{ALM}$, then $(\mathbf{K}_{*}, \preccurlyeq)$ does. 
\end{theorem}

	\begin{proof} The furthermore part is immediate. For amalgamation, let $A, B, C \in \mathbf{K}_{*}$ be such that $C \preccurlyeq A, B$ and $A \cap B = C$ (without loss of generality). Then there exists basis $S'$ for $A$ and $T'$ for $B$ such that $S = S'\cap A$ and $T = T' \cap B$ are bases for $C$. Thus, there exists $g \in C$ such that $gTs^{-1} = S$, and so $gT's^{-1} = S''$ is a basis for $B$ such that $S' \cap S'' = S$. Hence, any amalgam for $(S,E) \leq (S', E), (S'', E)$ is an amalgam for $C \preccurlyeq A, B$.
	Items (1) and (2) of Definition \ref{def_AEC} are clear. We prove (3). Let $A \preccurlyeq B \preccurlyeq C$. Then there exists a basis $S'$ for $B$ such that $S = S' \cap A$ is a basis for $A$, and a basis $T''$ for $C$ such that $T' = T'' \cap B$ is a basis for $B$. Thus, because of strong rigidity, there exists $g \in B$ such that $S' = gT'g^{-1}$, and so $S'' = gT''g^{-1}$ is a basis for $C$ containing $S$, i.e. $A \preccurlyeq C$.
\smallskip

\noindent	
	We prove (4.1) and (4.2). Let $({A}_i)_{i < \delta}$ be an increasing continuous $\preccurlyeq$-chain. Using strong rigidity, without loss of generality we can assume that $(\Gamma_{A_i} = (S_i, E))_{i < \alpha}$ is an increasing continuous chain of graphs under the induced subgraph relation. Using the Universality Property for Coxeter groups (see e.g. \cite[pg. 3]{bjorner}) it is immediate to see that $\bigcup_{i < \delta} {A}_i = A$ is the Coxeter group of type $\bigcup_{i < \alpha}\Gamma_{A_i}$, and so $A \in  \mathbf{K}$. This establishes (4.1) and (4.2) at once. 
\smallskip

\noindent
We prove (5). Let ${A} \preccurlyeq {C}$, $B \preccurlyeq {C}$ and ${A} \leq {B}$. Let $S''$ be a witness for ${A} \preccurlyeq {C}$ and $S = S'' \cap A$. Let also $T''$ be a witness for ${B} \preccurlyeq {C}$ and $T' = T'' \cap B$. Now, $S''$ and $T''$ are two bases for $C$ and so we can find $g \in C$ such that $S'' = gT''g^{-1}$, i.e. for every $s \in S''$ there exists $t_s \in T''$ such that $s = gt_sg^{-1}$. Let $a_1 \cdots a_k$ be a $T''$-normal form for $g$. Notice that $S \subseteq A \subseteq B$ and $S \subseteq S''$, and so for every $s \in S$ we have $s = gt_sg^{-1} \in B$. Thus
\begin{equation}\label{supports_equation}
sp(gt_sg^{-1}) \subseteq T',
\end{equation} 
where the support is taken in the basis $T''$. Let $a_{q_1} \cdots a_{q_n}$ be the subword of $a_1 \cdots a_k$ obtained by deleting all the occurrences of letters in $T'' - T'$. Then because of (\ref{supports_equation}) and Lemma \ref{subword_lemma} we have that $a_{q_1} \cdots a_{q_n} = h \in B$ is such that 
$$s = gt_sg^{-1} = ht_sh^{-1},$$
for every $s \in S$. Thus, $hT'h^{-1} = S'$ is a basis for $B$ such that $S \subseteq S'$, and so $A \preccurlyeq B$.
\end{proof}	

	\begin{lemma}\label{lemma_for_smooth} Let $B$ be a strongly rigid right-angled Coxeter group, and $T_0$ and $T_1$ bases for $B$. If $T_0 \cap T_1$ contains $P_4 = s_0Es_1Es_2Es_3$, $s_0$ is not adjacent to $s_2$, $s_1$ is not adjacent to $s_3$ and there is no $t \in T_1$ such that $s_0EtEs_1$, then $T_0 = T_1$.
\end{lemma}

	\begin{proof} Let $T_0$, $T_1$ and $P_4 = s_0Es_1Es_2Es_3$ be as in the statement of the theorem. Then there exists $g \in B$ such that $T_1 = gT_0g^{-1}$. Let $s \in P_4$, then $g s g^{-1} = s$, because otherwise we would have $s \neq g s g^{-1}$ both in $T_1$, contradicting the fact that $T_1$ is a basis for $B$ (cfr. \cite[pg. 5]{bourbaki}). Suppose now that there exists $t \in sp(g) - \{ s_0, s_1 \}$, where the support is taken in the basis $T_1$. Then $t$ commutes with $s_0$ because otherwise by Theorem \ref{normal_form} we have $s_0 \neq gs_0g^{-1}$. Similarly,  $t$ commutes with $s_1$ because otherwise $s_1 \neq gs_1g^{-1}$. Thus, $s_0EtEs_1$, which is a contradiction. Hence, $sp(g) \subseteq \{ s_0, s_1 \}$.
On the other hand, $s_0 \not\in sp(g)$ and $s_1 \not\in sp(g)$, because otherwise $s_2 \neq gs_2g^{-1}$ or  $s_3 \neq gs_3g^{-1}$. It follows that $g = 1$, i.e. $T_1 = T_0$.
\end{proof}

	\begin{theorem}\label{smoothness} Let $\mathbf{K}_{*}$ be a class of strongly rigid right-angled Coxeter groups such that for every $A \in \mathbf{K}_{*}$ we have that $\Gamma_A$ is triangle-free. Suppose further that whenever $A \preccurlyeq B \in \mathbf{K}_{*}$ and $T$ is a basis for $B$ such that $S = T \cap A$ is a basis for $A$, then the basis $S$ contains a copy of $P_4 = s_0Es_1Es_2Es_3$ such that $s_0$ is not adjacent to $s_2$ and $s_1$ is not adjacent to $s_3$. Then $(\mathbf{K}_{*}, \preccurlyeq)$ satisfies the Smoothness Axiom and it has finite character.
\end{theorem}

\begin{proof} We show that $(\mathbf{K}_{*}, \preccurlyeq)$ is smooth. Let $({A}_i)_{i < \alpha}$ be an increasing continuous $\preccurlyeq$-chain such that each ${A}_i \preccurlyeq {B}$. Using strong rigidity, without loss of generality we can assume that $(\Gamma_{A_i} = (S_i, E))_{i < \alpha}$ is an increasing continuous chain of graphs under the induced subgraph relation, and that there are $(T_i)_{i < \alpha}$ bases for $B$ such that $T_i \cap A_i = S_i$, for every $i < \alpha$. Let $i < \alpha$, then using the assumption of the theorem for $T_i$ and $S_0$ we have that $T_0 \cap T_i$ contains $P_4 = s_0Es_1Es_2Es_3$, $s_0$ is not adjacent to $s_1$, $s_1$ is not adjacent to $s_3$ and there is no $t \in T_i$ such that $s_0EtEs_1$. Thus, by Lemma \ref{lemma_for_smooth}, we have that $T_i = T_0$. Hence, $\bigcup_{i < \alpha} S_i \subseteq T_0$, witnessing that $\bigcup_{i< \alpha} A_i \preccurlyeq B$.
\smallskip

\noindent
We show that $(\mathbf{K}_{*}, \preccurlyeq)$ has finite character. Suppose that $A \leq B$ and for every $X \subseteq_{fin} A$ there exists $\preccurlyeq$-embedding $f_X: A \rightarrow B$ such that $f \restriction X = id_X$. Let $S$ be a basis for $A$. For every $X \subseteq A$ we have $A \cong f_X(A)$, and so $f_X(S)$ is a basis for $f_X(A)$. It follows that:
\begin{equation}\tag{$\star$}
\text{$\forall X \subseteq_{fin} S$, $\exists \, T_X$ basis of $B$ such that $T_X$ extends $f_X(S)$ and $X \subseteq T_X$},
\end{equation} 
this is because $f_X(A) \preccurlyeq B$, of course. Fix $Y \subseteq_{fin} S$, then $f_Y(A) \preccurlyeq B$, and so using the assumption of the theorem for $T_Y$ and $f_Y(S)$ we get $P'_4 = s'_0Es'_1Es'_2Es'_3$ in $f_Y(S)$, such that $s'_0$ is not adjacent to $s'_2$ and $s'_1$ is not adjacent to $s'_3$. Let now $f^{-1}_Y(P'_4) = P_4 = s_0Es_1Es_2Es_3$. Then, noticing that $P_4 \subseteq S$, and recalling $(\star)$ and that $\Gamma_B$ is triangle free we have that $T_{P_4}$ is a basis of $B$ such that $s_0$ is not adjacent to $s_1$, $s_1$ is not adjacent to $s_3$ and there is no $t \in T_{P_4}$ such that $s_0EtEs_1$. Thus, by Lemma \ref{lemma_for_smooth}, for every $P_4 \subseteq X \subseteq_{fin} S$ we have that $T_X = T_{P_4}$. Hence, for every $X \subseteq_{fin} S$ we have $X \subseteq T_{P_4}$, and so $S \subseteq T_{P_4}$, i.e. $A \preccurlyeq B$.
\end{proof}

	Let $\mathbf{K}'_{*}$ be a class of graphs such that $(\mathbf{K}'_{*}, \leq)$ is an $\mathrm{AEC}$ with $\mathrm{AP}$, $\mathrm{JEP}$ and $\mathrm{ALM}$. Suppose that $\mathbf{K}_{*} = \left\{ A \in \mathbf{K} \!: \Gamma_A \!\in \mathbf{K}'_{*} \right\}$ is a class of strongly rigid right-angled Coxeter groups, and that $(\mathbf{K}_{*}, \preccurlyeq)$ is also an $\mathrm{AEC}$ (and thus, by Theorem \ref{almost_AEC_th}, it has $\mathrm{AP}$, $\mathrm{JEP}$ and $\mathrm{ALM}$). Notice that under these conditions, modifying a little the construction of $\mathfrak{M}(\mathbf{K}_{*}, \preccurlyeq)$ we can assume that $\Gamma_{\mathfrak{M}(\mathbf{K}_{*}, \preccurlyeq)} = \mathfrak{M}(\mathbf{K}'_{*}, \leq)$. In the following theorem we will use this assumption crucially.

	\begin{theorem}\label{tameness} Let $\mathbf{K}'_{*}$ be a class of graphs such that $(\mathbf{K}'_{*}, \leq)$ is an $\mathrm{AEC}$ with $\mathrm{AP}$, $\mathrm{JEP}$ and $\mathrm{ALM}$. Suppose that $\mathbf{K}_{*} = \left\{ A \in \mathbf{K} \!: \Gamma_A \!\in \mathbf{K}'_{*} \right\}$ is a class of strongly rigid right-angled Coxeter groups, and that $(\mathbf{K}_{*}, \preccurlyeq)$ is also an $\mathrm{AEC}$ (and thus, by Theorem \ref{almost_AEC_th}, it has $\mathrm{AP}$, $\mathrm{JEP}$ and $\mathrm{ALM}$) with $\mathrm{LS}(\mathbf{K}_{*}, \preccurlyeq) = \omega$. Suppose further that for every $A \in \mathbf{K}$, $Aut(\mathfrak{M}/A) \leq Aut(\Gamma_{\mathfrak{M}})$. Then if $(\mathbf{K}'_{*}, \leq)$ is tame, so is $(\mathbf{K}_{*}, \preccurlyeq)$.
\end{theorem}

	\begin{proof} We show the tameness of $(\mathbf{K}_{*}, \preccurlyeq)$ for elements, the argument generalizes to tuples. Let $B \in \mathbf{K}_{*}$ and $a, b$ elements in $\mathfrak{M}(\mathbf{K}_{*}, \preccurlyeq)$, and suppose that $tp(a/B) \neq tp(b/B)$. Notice that for every $\alpha \in Aut(\Gamma_{\mathfrak{M}})$ the following are equivalent:
	\begin{enumerate}[(i)]
	\item $\alpha(a) = b$;
	\item $\alpha$ restricted to $sp(a)$ is a bijection from $sp(a)$ into $sp(b)$ such that if $a_1 \cdots a_k$ is a normal form for $a$, then $\alpha(a_1) \cdots \alpha(a_k)$ is a normal form form $b$;
	\item $\alpha$ restricted to $sp(a)$ is a bijection from $sp(a)$ into $sp(b)$, and there exists a normal form $a_1 \cdots a_k$ for $a$, such that $\alpha(a_1) \cdots \alpha(a_k)$ is a normal form form $b$.
\end{enumerate}
Now, if $|sp(a)| \neq |sp(b)|$ then for any countable $A \preccurlyeq B$ we have that $tp(a/A) \neq tp(b/A)$, since by assumption $Aut(\mathfrak{M}/A) \leq Aut(\Gamma_{\mathfrak{M}})$. Suppose then that $|sp(a)| = |sp(b)|$, fix a normal form $a_1 \cdots a_k$ for $a$ and let $\{ b_1^j \cdots b_k^j : j < n \}$ be the set of normal forms for $b$. For every $j < n$ we must have that 
	$$tp((a_i)_{0 < i \leq k}/\Gamma_B) \neq tp((b_i^j))_{0 < i \leq k}/\Gamma_B),$$
where types are in the sense of $(\mathbf{K}'_{*}, \leq)$ . In fact, otherwise there is 			    $$\alpha \in Aut(\mathfrak{M}(\mathbf{K}'_{*}, \leq)/ \Gamma_B) = Aut(\Gamma_{\mathfrak{M}(\mathbf{K}_{*}, \preccurlyeq)}/B)$$
such that $\alpha(sp(a)) = sp(b)$ and $\alpha(a_1) \cdots \alpha(a_k)$ is a normal form form $b$, and so $tp(a/B) = tp(b/B)$. Thus, by the tameness of $(\mathbf{K}'_{*}, \leq)$, for every $j < n$ there is countable $\Gamma_{A_j} \leq \Gamma_B$ such that
	$$tp((a_i)_{0 < i \leq k}/\Gamma_{A_j}) \neq tp((b_i^j))_{0 < i \leq k}/\Gamma_{A_j}).$$
Let $A \preccurlyeq B$ be such that $\bigcup_{j < n} A_j \subseteq A$. Then $tp(a/A) \neq tp(b/A)$. In fact, otherwise there exists $\alpha \in Aut(\Gamma_{\mathfrak{M}}/A)$ such that $\alpha(sp(a)) = sp(b)$ and $\alpha(a_1) \cdots \alpha(a_k)$ is a normal form form $b$, and so there exists $j<n$ and $\alpha \in Aut(\mathfrak{M}(\mathbf{K}'_{*}, \leq)/ \Gamma_{A_j})$ mapping $(a_i)_{0 < i \leq k}$ to $(b_i^j)_{0 < i \leq k}$, which is a contradiction.
\end{proof}
		
	Let $\mathbf{K}'_0$ be the class of graphs satisfying the following requirements:
\begin{enumerate}[(1)]
\item $\Gamma$ has the star property;
\item $\Gamma$ is star-connected;
\item $\Gamma$ is triangle-free;
\item $\Gamma$ contains $C_4$ (the cycle of length $4$) as an (induced) subgraph.
\end{enumerate}
Let then $\mathbf{K}_0 = \left\{ A \in \mathbf{K} : \Gamma_A \in \mathbf{K}'_0 \right\}$.
Notice that because of Corollary \ref{first_strong_rig_cor}, every $A \in \mathbf{K}_0$ is strongly rigid. We ask that $\Gamma$ contains $C_4$ instead of simply $P_4$ because $C_4$ has the star property, while $P_4$ does not. The fact that $C_4$ embeds as an induced subgraph in every structure in $\mathbf{K}'_0$ will be useful in proving joint embedding from amalgamation. We need a lemma before proving the main theorem of this section.

	\begin{lemma}\label{extension_lemma} Let $\Gamma$ be triangle-free and such that it contains $C_4$ as an induced subgraph. By induction on $i < \omega$, define $\Gamma_i$ such that:
	\begin{enumerate}[(i)]
	\item $\Gamma_0 = \Gamma$;
	\item $\Gamma_{i+1}$ is the extension of $\Gamma_i$ following the condition: for every $a \neq b \in \Gamma_i$ if $a$ is not adjacent to $b$, then add $c$ such that $N(c) = \{a, b \}$.
\end{enumerate}
Then $\Gamma \leq \bigcup_{i < \omega} \Gamma_i = \Gamma^* \in \mathbf{K}'_0$.
\end{lemma}

	\begin{proof} Obviously, $C_4 \leq \Gamma \leq \Gamma^*$ and $\Gamma^*$ is triangle-free. Regarding the star-property, let $a \neq b \in \Gamma^*$, we show that $st(a) \not\subseteq st(b)$. Assume $a, b \in \Gamma_i$. Then $\Gamma_{i+1} - \Gamma_i$ contains an element $x$ which is not adjacent to $a$ (since $C_4$ contains two adjacent vertices different from $a$).  Now, $\Gamma_{i+2} - \Gamma_{i+1}$ contains an element $c$ which is adjacent to $a$ and $x$, but not to $b$. Hence, $c \in st(a) - st(b)$, as wanted. Regarding star-connectedness, let $v \in \Gamma^*$ and $a \neq b \in \Gamma^* - st(v)$. Assume that $v,a,b \in \Gamma_i$. If $a$ and $b$ are adjacent in $\Gamma^*$, then they are connected in $\Gamma^* - st(v)$ (since $a \neq b \in \Gamma^* - st(v)$). If $a$ and $b$ are not adjacent in $\Gamma^*$, then they are not adjacent in $\Gamma_i$ either, and so at stage $\Gamma_{i+1}$ we have added $c$ such that $N(c) = \{a, b \}$, witnessing the connectedness of $a$ and $b$ in $\Gamma^* - st(v)$.
\end{proof}

	\begin{theorem}\label{first_finitary} $(\mathbf{K}_{0}, \preccurlyeq)$ is a finitary $\mathrm{AEC}$. 
\end{theorem}

	\begin{proof} As already noticed, because of Corollary \ref{first_strong_rig_cor}, every $A \in \mathbf{K}_0$ is strongly rigid. Furthermore, obviously $(\mathbf{K}'_0, \leq)$ is closed under limits and $\mathrm{LS}(\mathbf{K}'_{0}, \leq) = \omega$. Also, every $A \in \mathbf{K}_0$ is such that $\Gamma_A$ is triangle-free and contains $C_4$ as an induced subgraph, and so we can always a $P_4$ as in Theorem \ref{smoothness}. Thus, by Theorems \ref{almost_AEC_th} and \ref{smoothness}, in order to conclude it suffices to show that $(\mathbf{K}'_{0}, \leq)$ has joint embedding and amalgamation. Now, $C_4 \in \mathbf{K}'_{0}$ and $C_4$ embeds as an induced subgraph in every $A \in \mathbf{K}'_0$, thus it suffices to prove amalgamation. Let then $A, B, C \in \mathbf{K}'_{0}$ be such that $C \leq A, B$ and $A \cap B = C$ (without loss of generality), and consider $D = (A \cup B)^*$. Then is it easy to see that $D$ is an amalgam of $A$ and $B$ over $C$.
\end{proof}

	\begin{theorem}\label{firt_non_homog} \begin{enumerate}[(a)] 
	\item $(\mathbf{K}_{0}, \preccurlyeq)$ is not homogeneous.
	\item $(\mathbf{K}_0, \preccurlyeq)$ has the independence property, and thus it is unstable.
\end{enumerate}
\end{theorem}

	\begin{proof} We prove (a). Let $(t_i)_{i < \omega}$ and $(a_i)_{i < \omega}$ in  $\Gamma_{\mathfrak{M}}$, for $\mathfrak{M}$ the monster model of $(\mathbf{K}_{0}, \preccurlyeq)$, be such that the following conditions are met:
\begin{enumerate}[(i)]
\item $(t_i)_{i < \omega}$ is an independent set;
\item $(a_i)_{i < \omega}$ is an independent set;
\item for every $i < \omega$, $a_i$ is adjacent to $t_j$ iff $j \leq i$.
\end{enumerate}
Such sequences $(t_i)_{i < \omega}$ and $(a_i)_{i < \omega}$ can be found in $\Gamma_{\mathfrak{M}}$, e.g. using Lemma \ref{extension_lemma}. For $i < \omega$, let 
$$c_i = a_0 \cdots a_{i-1}t_ia_{i-1} \cdots a_{0}.$$
Then for every $X \subseteq_{fin} \omega$ we have $tp(t_X/\emptyset) = tp(c_X/\emptyset)$, as witnesses by the inner automorphism determined by $a_0 \cdots a_{k-1}$, for $k = max\{ i<\omega: i \in X\}$. On the other hand, $tp((t_i)_{i < \omega}/\emptyset) \neq tp((c_i)_{i < \omega}/\emptyset)$ because there is no automorphism of $\mathfrak{M}$ such that $t_i \mapsto c_i$ for every $i < \omega$, as this would contradict the strong rigidity of $\mathfrak{M}$, in fact no inner automorphism $gxg^{-1}$ (for $g \in \mathfrak{M}$) could serve as witness for this candidate automorphism, since $sp(g)$ is finite. We prove (b). Let 
$$P = \{ p \in S_2(\emptyset) : \forall a, b \in \mathfrak{M}, \text{ if } (a,b) \models p \text{ then } ab = ba \},$$ $\alpha < \card{\mathfrak{M}},$ and $(t_i)_{i < \alpha}$ and $(a_X)_{X \subseteq \alpha}$ in  $\Gamma_{\mathfrak{M}}$ be such that the following conditions are met:
\begin{enumerate}[(i)]
\item $(t_i)_{i < \alpha}$ is an independent set;
\item $(a_X)_{X \subseteq \alpha}$ is an independent set;
\item for every $X \subseteq \alpha$, $a_X$ is adjacent to $t_i$ iff $i \in X$.
\end{enumerate}
Such sequences $(t_i)_{i < \alpha}$ and $(a_X)_{X \subseteq \alpha}$ can be found in $\Gamma_{\mathfrak{M}}$, e.g. using Lemma \ref{extension_lemma}. Evidently, $tp(a_Xt_i/\emptyset) \in P$ if and only if $i \in X$.
\end{proof}

	\begin{remark} The first configuration used in the proof of Theorem \ref{firt_non_homog} will play a crucial role also in the proof of Theorem \ref{second_cate} (where a similar non-homogeneity result is proved). It is interesting to notice that the existence of this configuration (on tuples of elements), also known as the half-graph, can always be find in a definable way in the monster model of an unstable theory. Thus, we here have an analogy between non-homogeneity in AEC's and unstability in first-order theories. 
\end{remark}

	Given a graph $\Gamma = (V, E)$ we define the {\em barycentric subdivision} of $\Gamma$, denoted $\hat{\Gamma}$, to be the graph whose node set is the disjoint union of $V$ and $\{ c_{a,b} : a,b \in \Gamma, aEb \}$, and so that $N(c_{a,b}) = \{ a,b \}$ and, for $a \in V$, $N(a) = \{ c_{a,b} : b \in \Gamma, aEb \}$.
	Let $\mathbf{K}'_1$ be the class of barycentric subdivisions of clique with at least four elements, and 
	$\mathbf{K}_1 = \left\{ A \in \mathbf{K} : \Gamma(A) \in \mathbf{K}'_1 \right\}$.
	
	\begin{theorem}\label{K1_fin} $(\mathbf{K}_{1}, \preccurlyeq)$ is a finitary $\mathrm{AEC}$. 
\end{theorem}

	\begin{proof} Obviously, $(\mathbf{K}'_{1}, \leq)$ is closed under limits, it has $\mathrm{AP}$, $\mathrm{JEP}$ and $\mathrm{ALM}$, and $\mathrm{LS}(\mathbf{K}'_{1}, \preccurlyeq) = \omega$. Also, it is immediate to see that every $\Gamma \in \mathbf{K}'_1$ is star-connected, it has the star property and it contains $P_4$, and so, by Corollary \ref{first_strong_rig_cor}, every $A \in \mathbf{K}_1$ is strongly rigid. Finally, it is obvious from the definition that for any graph $\Gamma$ the graph $\hat{\Gamma}$ is bipartite (and thus triangle-free). Hence, by Theorems \ref{almost_AEC_th} and \ref{smoothness} we are done.
\end{proof}

	Let $\mathbf{K}''_{1}$ be the class of infinite structures in $\mathbf{K}'_{1}$. It is immediate to see that the class $\mathbf{K}''_{1}$ is axiomatizable by the following first-order theory $T$:
	\begin{enumerate}[(A)]
	\item there are infinitely many elements;
	\item every $x$ has either exactly two neighbours or at least three neighbours;
	\item if $x$ has exactly two neighbours $y$ and $z$, then $y$ and $z$ have at least three neighbours; 
	\item if $x$ has at least three neighbours, then each neighbour of $x$ has exactly two neighbours; 
	\item if $x \neq y$ have at least three neighbours, then there exists unique $z$ such that $xEzEy$.
\end{enumerate}

	\begin{proposition}\label{elim_quant} $T$ is complete and it is model complete.
\end{proposition}

	\begin{proof} Standard.
\end{proof}

	\begin{theorem}\label{K1_tame} $(\mathbf{K}_{1}, \preccurlyeq)$ is tame. 
\end{theorem}

	\begin{proof} Obviously, $(\mathbf{K}'_{1}, \leq)$ is an $\mathrm{AEC}$ with $\mathrm{AP}$, $\mathrm{JEP}$ and $\mathrm{ALM}$. Furthermore, by Lemma \ref{lemma_for_smooth}, for every $A \in \mathbf{K}_1$, $Aut(\mathfrak{M}/A) \leq Aut(\Gamma_{\mathfrak{M}})$. Thus, by Theorem \ref{tameness}, it suffices to show that $(\mathbf{K}'_{1}, \leq)$ is tame. Clearly, it suffices to prove tameness for the class $\mathbf{K}''_{1}$ of infinite structures in $\mathbf{K}'_{1}$. By Proposition \ref{elim_quant}, the class $\mathbf{K}''_{1}$ is axiomatizable by a complete first-order theory which is model complete. Thus, $(\mathbf{K}''_{1}, \leq) = (\mathbf{K}''_{1}, \preccurlyeq^*)$, where $\preccurlyeq^*$ denotes the elementary submodel relation of first-order logic, and clearly $(\mathbf{K}''_{1}, \preccurlyeq^*)$ is tame.
\end{proof}
	
	\begin{theorem}\label{first_cate} $(\mathbf{K}_{1}, \preccurlyeq)$ is uncountably categorical.
\end{theorem}

	\begin{proof} For uncountable $A, B \in \mathbf{K}_{1}$, letting $\Gamma_A = \hat{\Gamma}_0$ and $\Gamma_B = \hat{\Gamma}_1$ (for $\Gamma_0$ and $\Gamma_1$ cliques), we have $|A| = |B|$ iff $|\Gamma_A| = |\Gamma_B|$ iff $|\Gamma_0| = |\Gamma_1|$ iff $\Gamma_0 \cong \Gamma_1$ iff $\Gamma_A \cong \Gamma_B$ iff $A \cong B$.
\end{proof}

	\begin{corollary} $(\mathbf{K}_{1}, \preccurlyeq)$ is stable in every infinite cardinality.
\end{corollary}

	\begin{proof} This is a consequence of Theorems \ref{first_finitary}, \ref{first_cate} and \ref{cate_implies_stability}.
\end{proof}

	\begin{corollary} $\mathbf{K}_1 = Mod(\theta)$ for some $\theta \in L_{\omega_1, \omega}$. Furthermore, for every $A \preccurlyeq \mathfrak{M}$ we have $tp_{\omega_{1}, \omega}(a/A) = tp_{\omega_{1}, \omega}(b/A) \text{ iff } tp(a/A) = tp(b/A)$.
\end{corollary}

	\begin{proof} This is an immediate consequence of Theorems \ref{kueker1}, \ref{kueker2}, \ref{K1_fin}, \ref{K1_tame} and \ref{first_cate} together with the easy observation that $\mathbf{K}_1$ has at most countably many countable models.
\end{proof}

\section{Centered Right-Angled Coxeter Groups}\label{centered_groups}

	Theorems \ref{firt_non_homog} and \ref{first_cate} leave open the question of finding classes of right-angled Coxeter groups which are stable and non-homogeneous. In this section we use Corollary \ref{second_strong_rig_cor} to achieve this. Let $C^*$ be the graph on vertex set $\{ s,s'\} \cup \{ t_i : i < 4 \}$, with the following edge relation: $t_0Et_2Et_3Et_1Et_0$, $t_0EsEt_2$, $t_1Es'Et_3$ and $sEs'$ (cf. Figure \ref{myfigure}). For every $B_T \models T = Th(\mathbb{N}, s, 0)$ (where $s$ denotes the successor function) we define a graph $\Gamma_{B_{T}} = (C^* \cup B_T, E)$ in the following way (without loss of generality we assume $s^n(0) = n$ in $B_T$):
\begin{enumerate}[(1)]
\item $C^*$ is an induced subgraph of $\Gamma_{B_{T}}$;
\item $t_0$ and $t_2$ are adjacent to all the even numbers in $\mathbb{N}$;
\item $t_1$ is adjacent to $2$ and to all the odd numbers in $\mathbb{N}$;
\item $t_3$ is adjacent to all the odd numbers in $\mathbb{N}$;
\item $N(0) - \mathbb{N} = \{ 1 \}$ and, for every $0 < n \in \mathbb{N}$, $N(n) \cap \mathbb{N} = \{ n-1, n+1 \}$;
\item for every copy $Z$ of $\mathbb{Z}$ in $B_T$ and $b \in Z$, $N(b) \cap Z = \{ b-1, b+1 \}$;
\item for every copy $Z$ of $\mathbb{Z}$ in $B_T$, there exists $0_Z \in Z$ such that for every $\pm n_Z = 0_Z \pm n$ we have $N(\pm n_Z) \cap \mathbb{N} = \{ 0, ..., n \}$;
\item for every copy $Z$ of $\mathbb{Z}$ in $B_T$, $0_Z$, $-1_Z$ and $1_Z$ are adjacent to $t_3$.
\end{enumerate}
Let $\mathbf{K}'_2$ be the class of graphs $\Gamma'$ isomorphic to one of the graphs $\Gamma = (C^* \cup B_T, E)$ described above, and $\mathbf{K}_2 = \left\{ A \in \mathbf{K} : \Gamma_A \in \mathbf{K}'_2 \right\}$. 

	\begin{remark} The proof of the theorem below is straightforward, but the details are tiresome. We include them for completeness of exposition.
\end{remark}

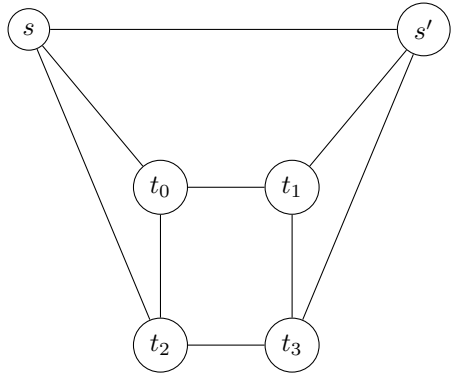
\begin{figure}
\begin{tikzpicture}
  [scale=.7,auto=left,every node/.style={circle, fill=none,draw}]
  \node (n-1) at (2.5,11)  {$s$};
  \node (n0)  at (10,11)  {$s'$};
  \node (n1)  at (5,8)  {$t_0$};
  \node (n2)  at (5,5)  {$t_2$};
  \node (n3)  at (7.5,5)  {$t_3$};
  \node (n4)  at (7.5,8)  {$t_1$};

  \foreach \from/\to in  {n-1/n0, n-1/n1, n-1/n2, n0/n4, n0/n3, n1/n2, n1/n4, n2/n3, n3/n4}
  \draw (\from) -- (\to);
\end{tikzpicture}\caption{The graph $C^*$.}
\end{figure}\label{myfigure}
	
	\begin{theorem}\label{second_finitary}\label{K2_fin} $(\mathbf{K}_{2}, \preccurlyeq)$ is a finitary $\mathrm{AEC}$.
\end{theorem}

	\begin{proof} Notice that for every $\Gamma = (C^* \cup B_T, E) \in \mathbf{K}'_2$, the structure $B_T$ can be recovered from $\Gamma$, and so $(\mathbf{K}'_{2}, \leq)$ is closed under limits, it has $\mathrm{AP}$, $\mathrm{JEP}$ and $\mathrm{ALM}$, and $\mathrm{LS}(\mathbf{K}'_{2}, \preccurlyeq) = \omega$. Thus, by Theorems \ref{almost_AEC_th} and \ref{smoothness} we are left to show that every $A \in \mathbf{K}_2$ is strongly rigid, and that the assumptions of Theorem \ref{smoothness} are met. The latter is immediate, since for every $A \in \mathbf{K}_{2}$ and basis $T$ of $A$, the elements $sEs'Et_1Et_0 \in C^*$ (without loss of generality $C^*$ is in $T$) are such that $s$ is not adjacent to $t_1$, $s'$ is not adjacent to $t_0$ and there is no $t \in T$ such that $sEtEs'$. To see strong rigidity we use Corollary \ref{second_strong_rig_cor}. Let $A \in \mathbf{K}_{2}$, then the elements $s, s' \in C^*$ are such $s$ is adjacent to $s'$ in $\Gamma_A$, and $st(s) \cup st(s') = C^*$ is finite and star-connected, since for every $x \in C^*$ we have $C^* - st(x) = \{ y, z \}$, for some $y,z \in C^*$ such that $y$ is adjacent to $z$. Furthermore, clearly for every $v \in \Gamma_A$ there exists $v \neq a \in st(s) \cup st(s')$ such that $v$ is not adjacent to $a$. Thus, we are left to show that $\Gamma_A$ is star-connected and it has the star property. For ease of notation, we assume that in $\Gamma_A$ the copies of $C^*$ and $\mathbb{N}$ are actually $C^*$ and $\mathbb{N}$ (we already did this for $C^*$ above). Also, we denote by $Z$, $Z'$, etc. the copies of $\mathbb{Z}$ possibly present in $\Gamma_A$. We first show that $\Gamma_A$ has the star property. Let $a \neq b \in \Gamma_A$.
\newline {\bf Case 1.} $a,b \in C^*$. Clear.
\newline {\bf Case 2.} $a,b \in \mathbb{N}$. Without loss of generality $a < b$. If $a = 0$ and $b = 1$, then $t_0 \in st(a) - st(b)$, and $t_1 \in st(b) - st(a)$. If $a=0$ and $b=2$, then $3 \in st(2) - st(0)$ and $t_1 \in st(2) - st(0)$. If $a = 0$ and $b > 2$, then $b+1 \in st(b) - st(0)$ and $1 \in st(0) - st(b)$. If $a > 0$, then $b+1 \in st(b) - st(a)$ and $a-1 \in st(a) - st(b)$.
\newline {\bf Case 3.} $a,b \in Z$. Without loss of generality $a < b$. We have $b+1 \in st(b) - st(a)$ and $a-1 \in st(a) - st(b)$.
\newline {\bf Case 4.} $a \in C^*$ and $b \in \mathbb{N}$. If $a = s$ or $a = s'$, then $a$ is not adjacent to $b$. Let $a = t_i$, for $i < 4$. If $i$ is even and $b$ is odd, then $a$ is not adjacent to $b$. If $i$ is odd and $b$ is even, then $a$ is not adjacent to $b$. If $i$ is even and $b$ is even, then $s \in st(a) - st(b)$ and $b + 1 \in st(b) - st(a)$. If $i$ is odd and $b$ is odd, then $s' \in st(a) - st(b)$ and $b + 3 \in st(b) - st(a)$ (in the case $b = 1$ we have $1+1 = 2Et_1$). 
\newline {\bf Case 5.} $a \in C^*$ and $b \in Z$. In this case $a$ is not adjacent to $b$, unless $a = t_3$ and $b \in \{ 0_Z, -1_Z, 1_Z \}$. In this case we have $t_2 \in st(a) - st(b)$ and $3_Z \in st(b) - st(a)$.
\newline {\bf Case 6.} $a \in \mathbb{N}$ and $b \in Z$. Let $b = \pm n_Z$. If $a > n$, then $a$ is not adjacent to $b$. If $0 < a \leq n$, then $n+2 \in st(a) -st(b)$ and $a-1 \in st(b) - st(a)$. If $a = 0 = n$, then $n+2 \in st(a) -st(b)$ and $t_3 \in st(b) - st(a)$. If $a = 0 < n$, then $n+2 \in st(a) -st(b)$ and $1 \in st(b) - st(a)$.
\newline {\bf Case 7.} $a \in Z$ and $b \in Z'$. In this case $a$ is not adjacent to $b$.
\smallskip 

\noindent
We now show that $\Gamma_A$ is star-connected. Let $v \in \Gamma_A$ and $a \neq b \in \Gamma_A - st(v)$.
\newline {\bf Case A.} $v \in C^*$. If $v = s$ or $v = s'$, then it is clear that $a$ is connected to $b$. Suppose then that $v = t_i$, for $i < 4$.
\newline {\bf Case A.1.} $a,b \in C^*$. Clear.
\newline {\bf Case A.2.} $a,b \in \mathbb{N}$. Then either both $a$ and $b$ are even, or both $a$ and $b$ are odd. In either cases we are fine. 
\newline {\bf Case A.3.} $a,b \in Z$. If $i < 3$, we have $(\Gamma_A - st(v)) \cap Z = Z$. If $i = 3$, we have $(\Gamma_A - st(v)) \cap Z = Z - \{0_Z, -1_Z, 1_Z\}$. In either cases we are fine.
\newline {\bf Case A.4.} $a \in C^*$ and $b \in \mathbb{N}$. If $a = s$ or $a = s'$, then $a$ is not adjacent to $b$. Suppose then that $a \not\in \{s,s'\}$. If $i$ is even (for $v = t_i$, remember) then $a = t_j$ is such that $j$ is odd and $b$ is odd, and so we are fine. If $i$ is odd, then $a = t_j$ is such that $j$ is even and $b$ is even, and so we are fine.  
\newline {\bf Case A.5.} $a \in C^*$ and $b \in Z$. If $i$ is even, then we can find an odd number $n \in \Gamma_A - st(v)$ that connects what is left of $C^*$ to $nEn_ZEb$. If $i$ is odd, then we can find an even number that does the same. 
\newline {\bf Case A.6.} $a \in \mathbb{N}$ and $b \in Z$. If $i$ is even, then we can find an odd number $n \in \Gamma_A - st(v)$ such that $aEnEn_ZEb$. If $i$ is odd, then we can find an even number that does the same. 
\newline {\bf Case A.7.} $a \in Z$ and $b \in Z'$. If $i$ is even, then we can find an odd number $n \in \Gamma_A - st(v)$ such that $aEn_ZEnEn_{Z'}Eb$. If $i$ is odd, then we can find an even number that does the same. 
\newline {\bf Case B.} $v \in \mathbb{N}$. 
\newline {\bf Case B.1.} $a,b \in C^*$. If $v = 2$, then $(\Gamma_A - st(v)) \cap C^* = sEs'Et_3$. If $v \neq 2$, then $(\Gamma_A - st(v)) \cap C^*$ is either $sEs'Et_1Et_3Es'$ or $s'EsEt_0Et_2Es$. In all of these cases we are fine. 
\newline {\bf Case B.2.} $a,b \in \mathbb{N}$. If $v = 0$, then $(\Gamma_A - st(v)) \cap \mathbb{N} = \{ 1 \}$, and so this case is not possible, since we are assuming that $a \neq b$. Thus, we must have that $v = n \neq 0$, and so $n-1 = aEb = n+1$.
\newline {\bf Case B.3.} $a,b \in Z$. If $v = 0$ or $v = 1$, then $(\Gamma_A - st(v)) \cap Z$ is either  $\emptyset$ or  $\{ 0_Z \}$, and so this case is not possible, since we are assuming that $a \neq b$. If $v = 2$, then $(\Gamma_A - st(v)) \cap Z = \{ 0_Z, -1_Z, 1_Z \}$, but $-1_ZEt_3E0_ZEt_3E1_Z$ and $t_3 \in \Gamma_A - st(2)$, and so we are fine. If $v > 2$, then $(\Gamma_A - st(v)) \cap Z = \{ -(v_Z-1), ..., 0_Z , ..., (v_Z-1)\}$ is ``long enough'', and so it is connected.
\newline {\bf Case B.4.} $a \in C^*$ and $b \in \mathbb{N}$. If $a = s$ or $a = s'$, then $a$ is not adjacent to $b$. Suppose then that $a \not\in \{s,s'\}$. If $v$ is odd, then $b$ is even and $a \in \{ t_0, t_2\}$, and so we are fine. If $v$ is even, then $b$ is odd and $a \in \{ t_1, t_3 \}$, and so we are fine. 
\newline {\bf Case B.5.} $a \in C^*$ and $b \in Z$. If $v = 0$, then $(\Gamma_A - st(v)) \cap Z = \emptyset$, and so this case is not possible. If $v = 1$, then $(\Gamma_A - st(v)) \cap Z = \{ 0_Z \}$ and $(\Gamma_A - st(v)) \cap C^* = \{ t_0, s, s', t_2 \}$, but then we are fine because $0_ZE0Et_0Et_2EsEs'$ and $0 \in \Gamma_A - st(1)$. If $v = 2$, then $(\Gamma_A - st(v)) \cap Z = \{ 0_Z, -1_Z, 1_Z \}$ and $(\Gamma_A - st(v)) \cap C^* = \{ t_3, s, s' \}$, but $-1_ZEt_3E0_ZEt_3E1_Z$, and so we are fine. If $v > 2$, then $(\Gamma_A - st(v)) \cap Z$ is connected, and so we can connect it to what is left of $C^*$ via $v-1 \not\in (\Gamma_A - st(v))$.
\newline {\bf Case B.6.} $a \in \mathbb{N}$ and $b \in Z$. If $v = 0$, then $(\Gamma_A - st(v)) \cap Z = \emptyset$, and so this case is not possible. If $v = 1$, then $(\Gamma_A - st(v)) \cap Z = \{ 0_Z \}$ and $(\Gamma_A - st(v)) \cap \mathbb{N} = \{ 0, 2 \}$, and $0_ZE0E2$. If $v = 2$, then $(\Gamma_A - st(v)) \cap Z = \{ 0_Z, -1_Z, 1_Z \}$ and $(\Gamma_A - st(v)) \cap \mathbb{N} = \{ 1, 3 \}$, and for $x \in \{ 0_Z, -1_Z, 1_Z \}$ and $y \in \{ 1, 3 \}$ we have $xEt_3Ey$, and so we are fine because $t_3 \in \Gamma_A -st(2)$. If $v > 2$, then $(\Gamma_A - st(v)) \cap Z$ is connected, and $v_Z-1Ev-iEv+1$.
\newline {\bf Case B.7.} $a \in Z$ and $b \in Z'$. If $v = 0$, then $(\Gamma_A - st(v)) \cap Z = \emptyset = (\Gamma_A - st(v)) \cap Z'$, and so this case is not possible. If $v = 1$, then $(\Gamma_A - st(v)) \cap Z = \{ 0_Z \}$ and $(\Gamma_A - st(v)) \cap Z = \{ 0_{Z'} \}$, and $0_ZE0E0_{Z'}$, and so we are fine because $0 \not\in \Gamma_A - st(1)$. If $v = 2$, then $(\Gamma_A - st(v)) \cap Z = \{ 0_Z, -1_Z, 1_Z \}$ and $(\Gamma_A - st(v)) \cap Z' = \{ 0_{Z'}, -1_{Z'}, 1_{Z'} \}$, and for $x \in \{ 0_Z, -1_Z, 1_Z \}$ and $y \in \{ 0_{Z'}, -1_{Z'}, 1_{Z'} \}$ we have $xEt_3Ey$, and so we are fine because $t_3 \in \Gamma_A -st(2)$. If $v > 2$, then $(\Gamma_A - st(v)) \cap Z$ and $(\Gamma_A - st(v)) \cap Z'$ are connected, and so we can connect them via $v-1 \not\in \Gamma_A - st(v)$.
\newline {\bf Case C.} $v \in Z$. Let $v = \pm n_Z$.
\newline {\bf Case C.1.} $a,b \in C^*$. We have $(\Gamma_A - st(v)) \cap C^* \subseteq C^* - \{ t_3 \}$, and so we are fine.
\newline {\bf Case C.2.} $a,b \in \mathbb{N}$. We have $(\Gamma_A - st(v)) \cap \mathbb{N} = \{ m \in \mathbb{N} : n < m \}$, and so we are fine.
\newline {\bf Case C.3.} $a,b \in Z$. In this case $v-1 = aEb = v+1$.
\newline {\bf Case C.4.} $a \in C^*$ and $b \in \mathbb{N}$. We have $(\Gamma_A - st(v)) \cap C^* \subseteq C^* - \{ t_3 \}$ and $(\Gamma_A - st(v)) \cap \mathbb{N} = \{ m \in \mathbb{N} : n < m \}$, and so we are fine.
\newline {\bf Case C.5.} $a \in C^*$ and $b \in Z$. We have $(\Gamma_A - st(v)) \cap C^* \subseteq C^* - \{ t_3 \}$ and $(\Gamma_A - st(v)) \cap Z = \{ \pm n_Z-1, \pm n_Z+1 \}$. Now,
$\pm n_Z-1E\pm n_Z+1En+1Et_j$, for some $j < 3$, and so we are fine  because $n+1 \in \Gamma_A - st(v)$ and $(\Gamma_A - st(v)) \cap C^*$ is connected.
\newline {\bf Case C.6.} $a \in \mathbb{N}$ and $b \in Z$. We have $(\Gamma_A - st(v)) \cap \mathbb{N} = \{ m \in \mathbb{N} : n < m \}$ and $(\Gamma_A - st(v)) \cap Z = \{ \pm n_Z-1, \pm n_Z+1 \}$. Now, $\pm n_Z-1E\pm n_Z+1En+1$, and $n+1$ is connected in $\Gamma_A - st(v)$ to every $x \in \{ m \in \mathbb{N} : n < m \}$.
\newline {\bf Case C.7.} $a \in Z$ and $b \in Z'$. We have $(\Gamma_A - st(v)) \cap Z = \{ \pm n_Z-1, \pm n_Z+1 \}$ and $(\Gamma_A - st(v)) \cap Z' = Z'$, and so $\pm n_Z-1E\pm n_Z+1En+1En_{Z'}+1$, and $n_{Z'}+1$ is connected in $\Gamma_A - st(v)$ to every $x \in Z'$.
\end{proof}

	\begin{theorem}\label{K2_tame} $(\mathbf{K}_{2}, \preccurlyeq)$ is tame. 
\end{theorem}

	\begin{proof} Obviously, $(\mathbf{K}'_{2}, \leq)$ is an $\mathrm{AEC}$ with $\mathrm{AP}$, $\mathrm{JEP}$ and $\mathrm{ALM}$. Furthermore, by Lemma \ref{lemma_for_smooth}, for every $A \in \mathbf{K}_2$, $Aut(\mathfrak{M}/A) \leq Aut(\Gamma_{\mathfrak{M}})$. Thus, by Theorem \ref{tameness}, it suffices to show that $(\mathbf{K}'_{2}, \leq)$ is tame. We show tameness of $(\mathbf{K}'_{2}, \leq)$ for elements, the argument generalizes to tuples. Let $B \in \mathbf{K}'_{2}$ and assume that in $B$ the copies of $C^*$ and $\mathbb{N}$ are actually $C^*$ and $\mathbb{N}$. Let $a, b \in \mathfrak{M}(\mathbf{K}'_{2}, \leq) - B$. Then $a$ and $b$ lie in some of the copies of $\mathbb{Z}$ not in $B$, say $a$
is in $Z$ and $b$ is in $Z'$. Let $a = 0_Z \pm n$ and $b = 0_{Z'} \pm m$. Notice that there
is $\alpha \in Aut(\mathfrak{M}/B)$ mapping $a$ to $b$ iff $n =
m$ iff there is $\alpha \in Aut(\mathfrak{M}/C^* \cup \mathbb{N})$ mapping $a$ to $b$. In fact,
for every copy $Z''$ of  $\mathbb{Z}$ we have that $0_{Z''} \pm n$ is adjacent
to exactly $n+1$ elements from $\mathbb{N}$. It follows that $tp(a/B) \neq tp(b/B)$ iff $tp(a/C^* \cup \mathbb{N}) \neq tp(b/C^* \cup \mathbb{N})$, and so $(\mathbf{K}'_{2}, \leq)$  is tame, because $C^* \cup \mathbb{N} \leq B$.
\end{proof}

	\begin{theorem}\label{second_cate} \begin{enumerate}[(a)] 
	\item $(\mathbf{K}_{2}, \preccurlyeq)$ is not homogeneous.
	\item $(\mathbf{K}_{2}, \preccurlyeq)$ is uncountably categorical.
\end{enumerate}
\end{theorem}

	\begin{proof} The proof of (a) is as in the proof of Theorem \ref{firt_non_homog}(a). In fact letting $t'_i = i$ and $a_i = i_Z$, for $Z$ a copy of $\mathbb{Z}$, we have that the argument used in the proof of Theorem \ref{firt_non_homog}(a) works also in this case (where the role of the $t_i$'s there is played by the $t'_i$'s here). Uncountable categoricity is also immediate, since for $C,D \in \mathbf{K}_{2}$ we have $(A \cup B_T, E) = \Gamma_C \cong \Gamma_D = (A' \cup B'_T, E')$ iff $B_T \cong B'_T$ (in the language $\{ 0, s \}$), and $T = Th(\mathbb{N}, s, 0)$ is well-known to be uncountably categorical.
\end{proof}

	\begin{corollary} $(\mathbf{K}_{2}, \preccurlyeq)$ is stable in every infinite cardinality.
\end{corollary}

	\begin{proof} This is a consequence of Theorems \ref{second_finitary}, \ref{second_cate} and \ref{cate_implies_stability}.
\end{proof}

	\begin{corollary} $\mathbf{K}_2 = Mod(\theta)$ for some $\theta \in L_{\omega_1, \omega}$. Furthermore, for every $A \preccurlyeq \mathfrak{M}$ we have $tp_{\omega_{1}, \omega}(a/A) = tp_{\omega_{1}, \omega}(b/A) \text{ iff } tp(a/A) = tp(b/A)$.
\end{corollary}

\begin{proof} This is an immediate consequence of Theorems \ref{kueker1}, \ref{kueker2}, \ref{K2_fin}, \ref{K2_tame} and \ref{second_cate}, together with the easy observation that $\mathbf{K}_2$ has at most countably many countable models.
\end{proof}

	We conclude the paper with the following open problem.
	
	\begin{oproblem} Find combinatorial conditions on $\Gamma_A$ which are necessary and sufficient for the strong rigidity of an arbitrary right-angled Coxeter group $A$, and use them to develop the model theory of strongly rigid right-angled Coxeter groups, in the style of the present paper.
\end{oproblem}

\end{document}